\documentclass[a4paper]{amsart}
\usepackage[utf8]{inputenc}
\usepackage[utf8]{inputenc}
\usepackage{amsfonts,amsmath,dsfont,amssymb,amsthm,stmaryrd,bbm,color,comment,todonotes}
\usepackage{tikz} 
\usepackage{url}
\usepackage[margin=1.4in]{geometry}
\usepackage{hyperref}
\usepackage{comment}
\usepackage[nameinlink,capitalize]{cleveref}
\newtheorem{thm}{Theorem}[section]
\newtheorem{prop}[thm]{Proposition}
\newtheorem{defn}[thm]{Definition}
\newtheorem{lem}[thm]{Lemma}
\newtheorem{lemma}[thm]{Lemma}
\newtheorem{remark}[thm]{Remark}
\newtheorem{question}[thm]{Question}

\newtheorem{fact}[thm]{Fact}
\newtheorem{claim}[thm]{Claim}

\newtheorem{coro}[thm]{Corollary}

\newtheorem{example}[thm]{Example}

\newcommand{\NN}{\mathbb{N}}
\newcommand{\ZZ}{\mathbb{Z}}
\newcommand{\RR}{\mathbb{R}}
\newcommand{\Ss}[1]{\mathbb{S}_{#1}}

\providecommand{\Th}{\operatorname{Th}}

\newcommand{\TCcirc}{T_{\mathcal{L}_{G,C}}(\Ss{L})}
\newcommand{\Tcirc}{T_{\mathcal{L}_{G}}(\Ss{L})}
\newcommand{\LL}{\mathcal{L}}

\DeclareMathOperator{\Ind}{Ind}
\DeclareMathOperator{\dpr}{dpr}
\DeclareMathOperator{\bdn}{bdn}
\DeclareMathOperator{\tp}{tp}
\DeclareMathOperator{\qftp}{qftp}
\DeclareMathOperator{\cl}{cl}

\DeclareMathOperator*\lowlim{\underline{lim}}

\begin{document}
\title{The model theory of geometric random graphs}
\author{Omer Ben-Neria}
\address{Omer Ben-Neria, Einstein Institute of Mathematics, The Hebrew University of Jerusalem, Edmond
J. Safra Campus, Givat Ram, Jerusalem 91904, Israel}
\email{omer.bn@mail.huji.ac.il}
\author{Itay Kaplan}
\address{Itay Kaplan, Einstein Institute of Mathematics, The Hebrew University of Jerusalem, Edmond
J. Safra Campus, Givat Ram, Jerusalem 91904, Israel}
\email{kaplan@math.huji.ac.il}
\author{Tingxiang Zou}
\address{Tingxiang Zou, Institut für Mathematische Logik und 
Grundlagenforschung, Fachbereich Mathematik und Informatik, Universität Münster, Einsteinstrasse 62, 48149 Münster, Germany}
\email{tzou@uni-muenster.de}
\thanks{Ben-Neria would like to thank the Israel Science Foundation (ISF) for their support of this research (grant no. 1832/19). Kaplan would like to thank the Israel Science Foundation (ISF) for their support of this research (grants no. 1254/18 and 804/22). Zou is partially supported by DFG EXC 2044–390685587 and ANRDFG AAPG2019 (Geomod).}
\date{}
\begin{abstract}
We study the logical properties of infinite geometric random graphs, introduced by Bonato and Janssen. These are graphs whose vertex set is a dense ``generic'' subset of a metric space, where two vertices are adjacent with probability $p>0$ provided the distance between them is bounded by some constant number. We prove that for a large class of metric spaces, including circles, spheres and the complete Urysohn space, almost all geometric random graphs on a given space are elementary equivalent.  Moreover, their first-order theory can reveal geometric properties of the underlying metric space.
\end{abstract}


\maketitle

\section{Introduction}
Random geometric graphs (RGG) have been developed for modelling real-life networks such as social networks (see \cite{RGG}). These are finite graphs constructed by randomly placing nodes in some metric space, and connecting two nodes whenever the distance between them is within a certain range. An infinite version of this notion was introduced by Bonato and Janssen in \cite{BJ} as \emph{geometric random graphs}.
Given a metric space $(X,d)$ and some chosen parameter $\delta>0$. A geometric random graph on a countable subset $V\subseteq X$ is constructed as follows: two vertices $x,y\in V$ are adjacent independently with probability $p>0$ provided $d(x,y)<\delta$. 
The classical infinite random graph, or the Rado graph, can be seen as a geometric random graph on a metric space whose diameter is smaller than $\delta$. By re-scaling the metric, we can always take $\delta=1$. Properties of geometric random graphs vary depending on the underlying metric space. One of the properties studied in \cite{BJ,AS,BJQ} is the \emph{geometric Rado property}. A metric space $(X,d)$ is called \emph{geometrically Rado} on a countable subset $V\subseteq X$ if two samplings of geometric random graphs on $V$ are almost surely isomorphic. It is shown in \cite[Theorem 3.7]{BJ} that $(\mathbb{R}^d,\ell_{\infty})$ is geometrically Rado on any dense set $V$ such that no two distinct points in $V$ has integer distance. Similarly, Angel and Spinka \cite{AS} showed that any circle $\mathbb{S}_L$ of length $L\geq 2$ is geometrically Rado on almost all countable dense subsets, provided $L$ is rational. On the other hand, $\mathbb{R}^2$ with the Euclidean metric is not geometrically Rado on any countable set, in fact any two samplings of geometric random graphs are almost surely non-isomorphic \cite[Theorem 4.1]{BJ}. The same phenomenon happens for $\mathbb{S}_L$ with $L>2$ irrational \cite[Theorem 1.3]{AS}.

The purpose of this work is to extend the study of geometric random graphs to the perspective of first-order logic. Our focus shifts from isomorphisms of graphs to their first-order theories. The language of graphs $\mathcal{L}_G$ consists of a single binary relation symbol $E$. Given a graph $G = (V_G,E_G)$, let $\Th(G)$ be the set of all sentences $\phi$ in the language $\mathcal{L}_G$, which are satisfied by $G$. Instead of considering the geometric Rado property, we ask when two samplings of geometric random graphs $G_1, G_2$ in a metric space are elementarily equivalent (i.e., $\Th(G_1)=\Th(G_2)$). We will show that in all the metric spaces considered above, and even in a much broader context, any two samplings of geometric random graphs on almost all dense subsets are almost surely elementarily equivalent. Thus, this first-order theory (rather than the isomorphism type of the graphs) is an invariant of the metric space in question, and it is natural to ask about the extent of which this invariant classifies the metric space. 

Recall that the Rado graph is characterized by an axiom scheme saying that unless there is a logical contradiction, any finite configuration can occur. The same holds here for separable metric space with no isolated points. Rather than arguing from probability, we will use the following property, defined in \cite{BJ}, which holds almost surely for any sampling of geometric random graphs on dense subsets \cite[Theorem 2.1]{BJ}.

\begin{defn}
Let $(X,d)$ be a separable metric space without isolated points and $S\subseteq X$ be a dense\footnote{
In an attempt to capture the geometric structure of $(X,d)$, we are interested in the case where $S$ is a dense subset.} subset.  A graph $G$ with vertex set $S$ is said to have \emph{unit threshold} if any two adjacent vertices $u,v\in S$ satisfy that $d(u,v)<1$. The graph $G$ is \emph{geometrically existentially closed (g.e.c.)} if $G$ has unit threshold\footnote{In \cite{BJ} and \cite{AS}, a g.e.c.\ graph is not required to have unit threshold; we pack these two terminologies in one for convenience.} and for any vertex $s$ and any disjoint finite finite sets $A,B\subseteq S$ which are contained in $B_1(s)$, the open unit ball abound $s$, and any $\epsilon>0$, there exists a vertex $v\in S\setminus (A\cup B)$ such that $d(s,v)<\epsilon$ and $v$ is adjacent to every vertex in $A$ but no vertex in $B$.
\end{defn}
\begin{remark}
    Note that the g.e.c.\ condition requires the underlying dense set having no isolated points. We will always assume in this paper that the metric space is separable without isolated points. 
\end{remark}

We will show that there is a large family of metric spaces, including 
all Riemmanian submanifolds definable in some o-minimal expansion of the real field (such as the ones described above), for which g.e.c.\ graphs on any dense independent\footnote{See \cref{def:defining structure and ind set}.} set have the same first-order theory. Moreover, the theory can recover some geometric properties of the metric space. 

In \cref{sec:circle}, we examine and expand the results of \cite{AS} for geometric random graphs on circles. We prove in \cref{thm:Elementary.Equiv.for.Circle.Graphs} that given a circle $\mathbb{S}_L$ of length $L>2$, all g.e.c.\ graphs on integer-distance free dense subsets of $\mathbb{S}_L$ are elementarily equivalent. We denote the common theory as $T_{\mathcal{L}_G}(\mathbb{S}_L)$. We proceed in \cref{prop:circularorder} and show that for all $L>3$, $T_{\mathcal{L}_G}(\mathbb{S}_L)$ can define the circular order restricted to the underlying dense set. The assumption that the vertex set is integer-distance free is then explained by \cref{cor:IntergerDef,rem:why integer-distance free}.

In \cref{sec:elentaryEquivalence} we investigate under which assumptions on the metric space and constrains on dense subsets, all g.e.c.\ graphs on these subsets have the same first-order theory. We develop a general framework --- \emph{defining structures} for metric spaces --- and give a sufficient condition for dense sets --- \emph{independence} --- so that we obtain elementary equivalence, see \cref{thm:elementary equivalence for dense independent sets}. This context contains the circle case. We justify the condition required for dense sets by showing in \cref{thm:almostsureindependent} that under reasonable assumptions on the probability measure, an infinite i.i.d.\ sampling of points on the metric space is almost surely a dense independent set.

\cref{sec:Volumn} discusses under which conditions the theory of g.e.c.\ graphs on a metric space can detect the volume of this space. In the case of circles $\mathbb{S}_L$, the theory $T_{\mathcal{L}_G}(\mathbb{S}_L)$ recovers the length $L$ for $L>2$. We prove a generalization of this result in \cref{thm-volume} for metric spaces whose family of open balls have finite VC-dimension, equipped with a measure which assign balls with the same radius the same measure.

In \cref{sec:Burden} we further investigate geometric properties of a metric space that can be recovered by the first-order theory of their g.e.c.\ graphs. We prove in \cref{thm:bound on burden} that if $(X,d)$ is a Riemmanian manifold of dimension $r$ definable in some o-minimal theory, then the theory of g.e.c.\ graphs on dense independent sets is NTP$_2$ with burden bounded by $r$.

In the last section, we study g.e.c.\ graphs on the complete Urysohn space. We prove in \cref{cor:UrysohnRado} that all g.e.c.\ graphs on integer-distance free dense subsets of this space are isomorphic, hence the Urysohn space is geometrically Rado.

\begin{remark} \label{rem:Alex Kruckman}
The question about elementary equivalence has already been investigated\footnote{We were not aware of this when we published the first version of this paper online.} in \cite{KruckmanThesis} and \cite{ackerman2017properly}. They showed that geometric random graphs are \emph{ergodic structures}. Indeed, when there is an atomless probability measure $\mu$ on $X$, and one chooses the vertex set by an i.i.d.\ sampling on $X$ and constructs geometric random graphs on it, then almost surely these graphs have the same first-order theory. In fact, they proved something stronger in this case: these graphs have the same theory in any countable fragment of the infinitary logic $L_{\omega_1,\omega}$. See Example 3.1 in \cite{ackerman2017properly} or Example 2.1.1 in \cite{KruckmanThesis}.

Although there is an overlap with our result, the proofs are different and they apply to different contexts. For example, we work with the g.e.c.\ condition rather than arguing from probability and in our context the resulting theory does not depend on the measure but only the metric space. See also \cref{rem:Alex2}. 
\end{remark}

\subsection{Acknowlegement}
We would like to thank Alex Kruckman for pointing out his work with Ackerman, Freer and Patel on ergodic structures, see \cref{rem:Alex Kruckman}. We would also like to thank Weikun He for helpful discussions around \cref{sec:Volumn} and Dugald Macpherson for suggesting us to look at the geometric Rado property of the Urysohn space, which resulted in \cref{sec:Urysohn}.

\section{Circles: case study}\label{sec:circle}

In this section, we will prove that all g.e.c.\ graphs on a generic dense set in $\Ss{L}$ are elementary equivalent. Here generic means having no pairs of points of integer distances, which is used in \cite{AS} and \cite{BJ}. We will further explore the definable sets in these geometric random graphs with the aim of understanding why we need this generic condition on dense sets to obtain the geometric Rado property or elementary equivalence.

Let $\Ss{L}$ be the circle of length $L$ where $L \in [2,\infty)$. We regard the circle as a quotient group $(\RR/ L\ZZ, +_L)$. For notational simplicity, for $x\in \RR$, let $[x] = x+L\ZZ$. The circle is also a metric space with the metric being $d_L(a,b) := \min\{|x-y|,L-|x-y|\}$ where $x,y \in [0,L)$ are the (unique) elements such that $[x] = a,[y] = b$. 

In \cite[Theorem 1.3]{AS}, it has been shown that any two g.e.c.\ graphs on circles $\mathbb{S}_L$ with $L>2$ irrational are almost surely non-isomorphic. However, when $L\geq 2$ is rational, all g.e.c.\ graphs on $\mathbb{S}_L$ are isomorphic provided the underlying vertex sets are generic dense subsets of $\mathbb{S}_L$, where generic means \emph{integer-distance free}, i.e., there is no $n\in \ZZ$ such that $d_L(u,v)=n \pmod{L}$ for some $u\neq v$ in the vertex set. As we will see in \cref{cor:IntergerDef}, having integer distances in the vertex set (and the integer distances) can be identified by the first-order theory of the graph, and so two such g.e.c.\ graphs may not be isomorphic or even elementarily equivalent (see \cref{rem:why integer-distance free})

In the following, we will show that all g.e.c\ graphs on a fixed circle $\mathbb{S}_L$ with $L\in (2,\infty]$ are elementary equivalent provided the underlying vertex sets are dense and integer-distance free in $\mathbb{S}_L$ (regardless of the rationality of $L$). Note that when $L=2$, any g.e.c.\ graph on a dense set of $\mathbb{S}_L$ is the Rado graph, of which the first-order theory is well-understood.

     Define the circular order relation $C(-,-,-)$ on $\Ss{L}$ by $C(a,b,c)$ if $x<y<z$ or $y<z<x$ or $z<x<y$ where $x,y,z\in [0,L)$ are such that $[x]=a, [y]=b$ and $[z]=c$. Note that it respects the group operation. For $a,b \in \Ss{L}$, let $(a,b):=\{c\in \Ss{L}:C(a,c,b)\}$ be the interval between $a$ and $b$.

\begin{lemma} \label{lem:circular order lemma}
    Let $a,b,c\in \Ss{L}$. Let $d_1,d_2 \in [0,L)$ be the unique elements such that $b=a+_L [d_1]$ and $c = a+_L [d_2]$. Let $e = a+_L[d_2 - d_1]$.
    
    Then $C(a,b,c)$ is equivalent to $0<d_1<d_2$ which is further equivalent to $C(a,e,c)$.
\end{lemma}

\begin{proof}
The first equivalence is obvious. For the second one, suppose $d_1 <d_2$. Then $d_2-d_1 \in (0,L)$ and $d_2-d_1<d_2$, and so $C(a,e,c)$ holds by the first equivalence. On the other hand, suppose that $C(a,e,c)$ holds. Let $f \in [0,L)$ be such that $a+_L[f] = e = a+_L[d_2-d_1]$, i.e., (*) $f-d_2+d_1 \in L\ZZ$. By the first equivalence, we know that $0<f<d_2$. If $d_2 \leq d_1$, then $f+d_1-d_2\geq f >0$. Therefore, $0< f+d_1-d_2<d_1<L$ contradicting (*). 
\end{proof}

Let $G=(V,E)$ be a graph, for a vertex $v\in V$ and $k\in\mathbb{N}$, define $N_k(v)$ to be the neighbourhood of $v$ of distance $k$, that is $y\in N_k(v)$ iff $$\exists x_0,x_1,\ldots,x_k\left(v=x_0\land y=x_k\land\bigwedge_{0\leq i<k}x_i \mathrel{E} x_{i+1}\right).$$

We will use the Ehrenfeucht–Fra\"{i}ss\'{e} game to prove elementary equivalence.
  \begin{defn}[Ehrenfeucht–Fra\"{i}ss\'{e} game]
     Let $\mathcal{L}$ be a finite relational language, i.e., $\mathcal{L}$ contains only relational symbols.
  Given two $\mathcal{L}$-structures $ M$ and $ N$ and $n\in\NN$, define the Ehrenfeucht–Fra\"{i}ss\'{e} game $G_n( M, N)$ to be an $n$-round game between two players, Spoiler and Duplicator, played as follows. At round $1\leq i\leq n$, Spoiler plays first by picking either an element $a_i\in  M$ or an $b_i\in N$. Then Duplicator plays according to the following rule. If Spoiler picked an element of $ M$, then Duplicator picks an element $b_i\in N$, otherwise, Duplicator picks an element $a_i\in  M$. Duplicator wins the game if the partial map $f$ mapping $a_i$ to $b_i$ preserves all the relations in $\mathcal{L}$, i.e., $R(a_{i_1},\ldots,a_{i_k})\Longleftrightarrow R(b_{i_1},\ldots,b_{i_k})$ for any $k$-array relation $R\in\mathcal{L}$ and any $1\leq i_1,\ldots,i_k\leq n$.   
  \end{defn}

\begin{fact} \label{fac:EF games}
 \cite[Theorem 2.4.6]{marker}: $ M\equiv_{\mathcal{L}} N$ iff Duplicator has a winning strategy in $G_n( M, N)$ for all $n\in\NN$.
\end{fact}

    Let $T$ be an $\mathcal{L}$-theory. We say that $T$ has \emph{quantifier elimination} if for all $\mathcal{L}$-formula $\phi(x)$ there is a quantifier-free $\mathcal{L}$-formula $\psi(x)$ such that $$T\vDash \forall x(\phi(x)\leftrightarrow \psi(x)).$$
    
    \begin{fact}\cite[Corollary 3.1.6]{marker}
    Let $T$ be a complete $\mathcal{L}$-theory. 
    Suppose that for all quantifier-free formula $\psi(x,z)$ where $|z|=1$, if $ M, N\vDash T$, $ A$ a common substructure of $ M$ and $ N$, $\bar{a}$ a tuple of length $|x|$ in $A$, and $b\in  M$ such that $ M\vDash\phi(\bar{a},b)$, then there is $c\in N$ with $ N\vDash\phi(\bar{a},c)$. Then $T$ has quantifier elimination.
     \end{fact}
     The following corollary is standard but we keep it for completeness. 
     \begin{coro}\label{cor:qf-criteria}
     Let $\mathcal{L}$ be a countable language and $T$ be a complete $\mathcal{L}$-theory. Suppose there are finitely universal models (by this we mean models which realize every finitary type without parameters) $ M$ and $ N$ such that for any $a_1,\ldots,a_n\in  M$ and $b_1,\ldots,b_n\in N$ with $\qftp(a_1,\ldots,a_n)=\qftp(b_1,\ldots,b_n)$ and any $a\in  M$, there exists $b\in N$ such that  $$\qftp(a_1,\ldots,a_n,a)=\qftp(b_1,\ldots,b_n,b).$$ Then $T$ has quantifier elimination. 
     \end{coro}
     \begin{proof}
     Let $\psi(x,z)$ be a quantifier-free $\mathcal{L}$-formula, with $|x|=n$ and $|z|=1$. Let $ M', N'\vDash T$ and $ A$ a common substructure. Suppose we have $\bar{a}\in A^n$ and $b\in M'$ such that $ M'\vDash \psi(\bar{a},b)$. Consider the types $p$ and $q$, the type of $\bar{a}$ in $ M'$ and $ N'$ respectively. By finite universality, there are $\bar{t}=(t_1,\ldots,t_n)\in  M^n$ realising $p$ and $\bar{s}=(s_1,\ldots,s_n)\in N^n$ realising $q$. Note that $\exists z\psi(\bar{t},z)\in p$, hence there is $b'\in M$ with $ M\vDash \psi(\bar{t},b')$. By assumption, there is $c'\in N$ such that $\qftp(\bar{t},b')=\qftp(\bar{s},c')$. Thus, $ N\vDash \psi(\bar{t},c')$ and hence $\exists z\psi(\bar{s},z)\in q$. Therefore, there is $c\in N'$ satisfying $ N'\vDash\psi(\bar{a},c)$ as required.
     \end{proof}

\begin{thm}\label{thm:Elementary.Equiv.for.Circle.Graphs}
Let $L\in (2,\infty]$ and $S,S'\subseteq \mathbb{S}_L$ be dense integer-distance free subsets. Suppose $G=(S,E)$ and $G'=(S',E')$ are g.e.c.\ graphs on $S$ and $S'$ respectively. Then $G\equiv_{\mathcal{L}_G} G'$, where $\mathcal{L}_G=\{E\}$ is the graph language. Denote the common theory by $\Tcirc$.

Moreover, let $\mathcal{L}_{G,C}$ be the graph language expanded by the ternary relations $\{C_{z,t,k} : z,t,k\in \ZZ\}$ where $C_{z,t,k}(a,b,c)$ holds iff $C(a+_L [z],b+_L [t],c+_L [k])$. Then $G\equiv_{\mathcal{L}_{G,C}} G'$ and the common theory $\TCcirc$ has quantifier elimination.
\end{thm}

\begin{remark}
    If $L \leq 2$ the theorem is still true, but it is less interesting since in that case $\Tcirc$ is the theory of the random graph.
\end{remark}

\begin{proof}


Let $f:A\to B$ be a bijection with $A\subseteq S$ and $B\subseteq S'$. We call $f$ an \textbf{$n$-elementary map}, if the following holds:
\begin{itemize}
\item
$f$ is a graph isomorphism between $G\upharpoonright A$ and $G'\upharpoonright B$ .
\item
for all $a,b,c\in A$ and $z,t,k\in[-2^n,2^n]\cap\mathbb{Z}$, $C(a+_L [z],b+_L [t],c+_L [k])$ iff $C(f(a)+_L [z],f(b)+_L [t],f(c)+_L [k])$. 
\end{itemize}

\begin{claim}\label{cla:n-elemnetary map}
Suppose $f:A\to B$ is an $n$-elementary map with $n>0$, and $A,B$ are finite non-empty. Then for any $a\in G$, we can find $b'\in G'$ such that $f\cup\{(a,b')\}$ is an $(n-1)$-elementary map.
\end{claim}
\begin{proof}
Given $a\in G$, we want to find the corresponding $b'\in G'$. We may assume $a\not\in A$. As $A$ is finite, so is $A_n:=\{a'+_L [z]:z\in[-2^n,2^n]\cap\mathbb{Z},a'\in A\}$. Since $S$ is integer-distance free, $x+_L [z]\neq y+_L [z']$ for any $x\neq y\in S$ and $z,z'\in\mathbb{Z}$. In particular $a\not\in A_n$. Therefore, we can find $a_1,a_2\in A$
and $z_1,z_2 \in[-2^n,2^n]\cap\mathbb{Z}$ such that $C(a_1+_L [z_1],a,a_2+_L [z_2])$ and $A_n\cap (a_1+_L [z_1],a_2+_L [z_2])=\emptyset$. Let $A'\subseteq A$ be the set of vertices in $A$ that are adjacent to $a$. Note that $A'\subseteq (a-_L [1],a+_L [1])$. By the fact that $(a_1+_L [z_1],a_2+_L [z_2])$ is the smallest interval in $A_n$ that contains $a$, for any $b\in A'\cap(a-_L [1],a)$ it must be that $(b,b+_L [1])\supseteq (a_1+_L [z_1], a_2+_L [z_2])$. Similarly, for any $c\in A'\cap(a,a+_L[1])$, we have $(c-_L [1],c)\supseteq (a_1+_L [z_1], a_2+_L [z_2])$. 
Let $B'=f(A')$, $b_1=f(a_1)$ and $b_2=f(a_2)$. Let $B_n:=\{b+_L [z]:z\in[-2^n,2^n]\cap\mathbb{Z},b\in B\}$.
Since $f$ is an $n$-elementary map, we know that $(b_1+_L [z_1],b_2+_L [z_2])\cap B_n=\emptyset$ and for any $d\in B'$, either $(d,d+_L [1])\supseteq (b_1+_L [z_1], b_2+_L [z_2])$ or $(d-_L [1],d)\supseteq (b_1+_L [z_1], b_2+_L [z_2])$.
Therefore, for any $e\in (b_1+_L [z_1],b_2+_L [z_2])$, we have $B'\subseteq (e-_L [1],e+_L [1])$. Since $G'$ is g.e.c., by density we may pick $b'\in G'$ with $b'\in (b_1+_L [z_1],b_2+_L [z_2])$ and $b'$ is adjacent to everything in $B'$ but nothing in $B\setminus B'$. Clearly $f':=f\cup \{(a,b')\}$ is a graph isomorphism between $A\cup\{a\}$ and $B\cup\{b'\}$. Now we verify the second condition of $f'$ being an $(n-1)$-elementary map. There are three cases to check, depending on the number of occurrence of $a$ in $C(x+_L[z],u+_L[t],v+_L[k])$. First note that $C(a+_L[z],a+_L[t],a+_L[k])\Longleftrightarrow C(x+_L[z],x+_L[t],x+_L[k])$ for any $x\in \mathbb{S}_L$ and $z,t,k\in\mathbb{Z}$. And after rearranging variables in the circular order, we only need to consider the following two cases: $C(x+_L [z],a+_L [t],y+_L [k])$ and $C(a+_L [z],a+_L [t],y+_L [k])$ for $x,y\in A$.

Note that for all $x,y\in A$ and $z,t,k\in [-2^{n-1},2^{n-1}]\cap \mathbb{Z}$, the statement $$C(x+_L [z],a+_L [t],y+_L [k])\Longleftrightarrow C(f(x)+_L [z],b'+_L [t],f(y)+_L [k])$$ is equivalent to $$(*)~~C(x+_L [z-t],a,y+_L [k-t])\Longleftrightarrow C(f(x)+_L [z-t],b',f(y)+_L [k-t]).$$  As both $z-t$ and $k-t$ take value in $[-2^n,2^n]\cap\mathbb{Z}$, by our choice of $b'$, we can see that $(*)$ is true for all $x,y\in A$. We only need to show $$C(a+_L [z],a+_L [t],y+_L [k])\Longleftrightarrow C(b'+_L [z],b'+_L [t],f(y)+_L [k])$$  for all $y\in A$ and $z,t,k\in [-2^{n-1},2^{n-1}]\cap\mathbb{Z}$.

Let $d_1, d_2 \in [0,L)$ be such that $a+_L [d_1] = a+_L [t-z]$ and $a+_L [d_2] = y+_L [k-z]$. Then, we get that $C(a+_L [z],a+_L [t], y+_L [k])$ holds iff $C(a,a+_L [t-z],y+_L [k-z])$ which in turn holds iff $C(a, a+_L[d_2-d_1], a+_L[d_2])$ by \cref{lem:circular order lemma}. Since $a+_L[d_2 - d_1] = y+_L[k-t]$, we get $$C(a+_L [z],a+_L [t], y+_L [k])\Longleftrightarrow C(a,y+_L [k-t],y+_L [k-z]).$$ 
Similarly,  $C(b'+_L [z],b'+_L [t],f(y)+_L [k])\Longleftrightarrow C(b',f(y)+_L [k-t],f(y)+_L [k-z])$. 


By the fact that $k-z,k-t\in [-2^n,2^n]\cap\mathbb{Z}$ and our choice of $b'$, we do have $$C(a,y+_L [k-t],y+_L [k-z])\Longleftrightarrow C(b',f(y)+_L [k-t],f(y)+_L [k-z]).$$ Thus $C(a+_L [z],a+_L [t], y+_L [k])\Longleftrightarrow C(b'+_L [z],b'+_L [t], f(y)+_L [k])$ as desired.
\end{proof}
Now we can finish the proof of the first part of this theorem. In fact we will show that $G\equiv_{\mathcal{L}_{G,C}} G'$. It is enough to show that $G\equiv_{\mathcal{L}_{0}} G'$ for any finite $\LL_0 \subseteq \LL_{G,C}$, so let $\LL_0 = \LL_G \cup \{C_{z,t,k} : z,t,k \in [-2^m, 2^m] \}$ for some naturatl number $m$. We use \cref{fac:EF games}, so let $n$ be a fixed natural number and play the $n$-round Ehrenfeucht-Fra\"{i}ss\'{e} game on $G$ and $G'$ in $\LL_0$. We claim that Duplicator has a winning strategy. Namely, Duplicator can make sure that at any round $k\leq n$ of the game, the two sequences chosen from $S$ and $S'$ form an $(m+n-k)$-elementary map $f_k:A \to B$ where $A \subseteq S$ and $B \subseteq S'$ and $|A|=|B|=k$. Note that for any $a\in G$ and $a'\in G'$, as $S$ and $S'$ are on circles of the same length the map $a \mapsto a'$ is an $m+n-1$-elementary map. Thus, on the first round, Duplicator can choose any point. This strategy is possible to maintain thanks to \cref{cla:n-elemnetary map}. 
After $n$ rounds, the map $f_n$ is an $m$-elementary map. In particular, $f_n$ is a graph isomorphism preserving the relations in $\LL_0$, hence Duplicator wins. Hence, we have that $G \equiv_{\LL_{G,C}} G'$, and let $\TCcirc$ be the common complete theory.

We turn to quantifier elimination of the complete theory $\TCcirc$. Note that \cref{cla:n-elemnetary map} is expressible in first-order (without quantifiers), namely, for any $n,m \in \NN$, let $$\phi_{n,m}(x_0,\ldots,x_{m-1},y_0,\ldots, y_{m-1})$$ be the conjunction of $(x_i \mathrel{E} x_j) \leftrightarrow (y_i\mathrel{E} y_j)$ and $C_{z,t,k}(x_i,x_j,x_l) \leftrightarrow C_{z,t,k}(y_i,y_j,y_l)$ for any $z,t,k \in [-2^n, 2^n]$ and any $i,j,l<m$. Then \cref{cla:n-elemnetary map} implies that for $n>0$ and any $m$, \[\TCcirc \vDash \forall \bar{x}\bar{y} \forall x \exists y \left(\phi_{n,m}(\bar{x},\bar{y}) \to \phi_{n-1,m}(\bar{x}x,\bar{y}y)\right).\] This means that if $M$ is $\omega$-saturated, namely any consistent type over finitely many parameters has a realization, $\bar{a}, \bar{b} \in M^m$ and $\qftp(\bar{a}) = \qftp({b})$ then for any $c \in M$ there is $d \in M$ such that $\qftp(\bar{a}c) = \qftp(\bar{b}d)$. Any $\omega$-saturated model is also finitely universal hence by \cref{cor:qf-criteria} (applied with $M,N$ being the same $\omega$-saturated model), we are done.
\end{proof}


Note that when $L>2$ is a rational number, then up to logical equivalence in the theory $\TCcirc$, the set of formulas $\{C_{z,t,k}:z,t,k\in \ZZ\}$ is finite. By quantifier elimination, $T_{\mathcal{L}_{G,C}}(\Ss{L})$ has only finitely-many $n$-types for any natural number $n$. Thus, by the Ryll-Nardzewski theorem (see \cite[Theorem 7.3.1]{Hodges}), $\Tcirc$ is $\omega$-categorical, i.e., any two countable model of $\Tcirc$ are isomorphic. In particular, we recover the following result in \cite[Theorem 1.3(2)]{AS}.

\begin{coro} \label{cor:omega-cat}
Any two g.e.c.\ graphs on integer-distance free countable dense sets of $\mathbb{S}_L$ are isomorphic for $L\in\mathbb{Q}\cap(2,\infty)$.
\end{coro}

To complete the picture of \cref{thm:Elementary.Equiv.for.Circle.Graphs}, we will also show in $T_{\mathcal{L}_{G}}(\Ss{L})$ the ternary relations $\{C_{z,t,k} : z,t,k\in\ZZ\}$ are definable in the pure graph language with parameters.

For convenience, we will write $[x]+_L [y]$ as $x+y$ for all $x,y\in\RR$ for the rest of this section.

Let $L>2$ and $S\subseteq \Ss{L}$ be a dense subset. For $a,b\in S$, let $I(a,b)$ be the shorter arc (including end points) in $\Ss{L}$ between $a$ and $b$. Define $A[a,b]:=I(a,b)\cap S$. One of the main results in \cite[Theorem 1.1]{AS} is that $A[a,b]$ can be recovered from a g.e.c.\ graph $G=(S,E)$. Indeed their proof shows that $A[a,b]$ is uniformly definable in the graph language for adjacent $a,b\in S$. In fact when unit-balls are definable by the graph structure, for example when $L>4$ by the following remark, we can define intervals $A[a,b]$ easily as noticed by \cite[Lemma 4.4]{AS}. 

\begin{remark}\label{B1Definable}
Let $G=(S,E)$ be a g.e.c.\ graph on a dense set $S\subseteq \Ss{L}$. If $L>4$ then the open unit-balls $B_1(v)\cap S$ restricted to $S$ are definable by the following formula $$x\in N_2(v)\land\forall z ((x\mathrel{E}z)\to z\in N_2(v)).$$
\end{remark}

\begin{lem}
Let $L>3$ and $S\subseteq \Ss{L}$ be a dense subset. Let $B(-,-)$ be the binary relation on $S$ interpreted as $B(x,y)$ whenever $d_L(x,y)<1$. Then in the structure $(S, B(-,-))$ there is a formula $\phi(x;y,z)$ such that $\phi(S;a,b)=A[a,b]$ for all $a,b$ satisfying $B(a,b)$.
\end{lem}
\begin{proof} Let $$\phi(x;y,z):=\forall v(B(v,y)\land B(v,z)) \to B(x,v)).$$ It is easy to check this defines $A[y,z]$ when $L>3$ and $d_L(y,z)<1$.
\end{proof}

We will further show that $(S, B(-,-))$ defines (with parameters) the ternary relations $\{C_{z,t,k} : z,t,k\in \ZZ\}$ and the integer distance points. The set of parameters needed is a sequence of elements in $S$ which fixes the direction of the circular order in $\Ss{L}$.

\begin{defn}
Let $L>3$, a sequence $a_0,\ldots,a_{n_L}=a_0$ where $n_L:=\lfloor L\rfloor +1$ in $\Ss{L}$ is called \emph{an orienting loop} if $0<d_L(a_i,a_{i+1})<1$, $C(a_i,a_{i+1},a_{i+2})$ and $\{I(a_i,a_{i+1})\setminus\{a_i,a_{i+1}\}:i<n_L\}$ forms a partition of $\Ss{L}\setminus\{a_0,\cdots,a_{n_L-1}\}$.
\end{defn}

\begin{prop}\label{prop:circularorder}
For all $L>3$, let $S\subseteq \Ss{L}$ be a dense subset and $a_0,\ldots,a_{n_L}=a_0$ be an orienting loop in $S$. Then the structure $S_L:=(S,B(-,-),(a_i)_{i\leq n_L})$ defines the ternary relations $\{C_{z,t,k} : z,t,k\in \ZZ\}$ and the partial functions $\{f_{z}:=(a\mapsto a + z) : z\in\ZZ\}$.
\end{prop}
\begin{proof}
By the fact that $A[a,b]$ are uniformly definable in $S_L$ for all $a,b$ satisfying $B(a,b)$, we can define a \emph{uni-directional path of length n} to be a sequence $p_0,\ldots,p_n$ in $S$ such that $p_i\neq p_{i+1}$, $B(p_i,p_{i+1})$ and $A[p_i,p_{i+1}]\cap A[p_{i+1},p_{i+2}]=\{p_{i+1}\}$ for all $i$. With the orienting loop $a_0,\ldots,a_{n_L}$ we can define the circular order $C(-,-,-)$ as follows. Given $i \neq j< n_L$, let $D_{ij}:=\bigcup_{i\leq m<j}A[a_m,a_{m+1}]$ if $i<j$ and $D_{ij}:=\bigcup_{i\leq m<n_L, 0\leq m<j}A[a_m,a_{m+1}]$ if $i>j$ (note that $D_{ij}$ never equals $S$).
Let $T(x,y,z)$ be the formula expressing there are $i \neq j < n_L$ and a uni-directional path $p_0=a_i,p_1,\ldots, p_m=a_j$ with $m<n_L +3$ 
which contains $x,y,z$ as a subsequence (not necessarily consecutive) such that $\bigcup_{t<m} A[p_t,p_{t+1}]=D_{ij}$. Now we may define the circular order $C(x,y,z)$ by the formula $$C(x,y,z):=T(x,y,z)\vee T(y,z,x)\vee T(z,x,y).$$ Next we define in $S_L$ inductively the intervals $F[a+n,a+n+1):=\{x\in S: C(a+n,x,a+n+1)\text{ or }x=a+n\}$ for any $a\in S$ and $n\in\NN$ as follows. Define $$x\in F[a,a+1):=\left(x=a\vee (B(x,a)\land\exists w \neg B(w,a)\wedge (C(a,x,w))\right),$$ 
and for $n>0$ define $x\in F[a+n,a+n+1)$ as: $$x\not\in F[a+n-1,a+n)\land\exists z\in [a+n-1,a+n)\left (x\in F[z,z+1)\right).$$ Hence we may define the partial functions $y=f_z(x)$ for any integer $z>0$ by $$y\in F[x+z,x+z+1)\land \neg\exists y',y''\in F[x+z,x+z+1)(C(y',y,y'')).$$ 
Similarly, we can define the intervals $F(a-n,a-n+1]$ and the partial functions $f_z(x)$ for $z<0$. 
Together we can define the intervals $F[a-z,a-z+1)$ for all $z \in \ZZ$. 

Given $z,k\in\ZZ$ and $a,b\in S$, let $J_{z,k}(a,b):=(a+z\neq b+k)\land \neg C(a+z,b+k+1,a+z+1)$, which is a definable condition\footnote{Note that $a+z,b+k+1$ and $a+z+1$ might not be in $S$, hence $C(a+z,b+k+1,a+z+1)$ is not a priori definable.} defined by \begin{align*}
    &(F[a+z,a+z+1)\not\subseteq F[b+k,b+k+1)\cup F[b+k+1,b+k+2)) \\
    \lor & (F[a+z,a+z+1)=F[b+k+1,b+k+2)).
\end{align*}  When $J_{z,k}(a,b)$ holds, we may define $D(a+z,b+k):=\{x\in S: C(a+z,x,b+k)\}$ by the following formula $$\exists v\in F[a+z,a+z+1)\forall v'\in F[b+k,b+k+1) (C(v,x,v')).$$ Note that $J_{z,k}(a,b)$ is equivalent to $J_{z-t,k-t}(a,b)$ for all $t\in\ZZ$. Thus, assuming $J_{z,k}(a,b)$, we have $C_{z,t,k}(a,c,b)$ holds iff $c\in D(a+z-t,b+k-t)$. Since $L>3$ we have $C_{z,t,k}(a,c,b)$ implies $J_{z,k}(a,b)\lor J_{t,z}(c,a)\lor J_{k,t}(b,c)$. Now we may define $C_{z,t,k}(a,c,b)$ by
\begin{align*}
    &(J_{z,k}(a,b)\land c\in D(a+z-t,b+k-t))\\
    \lor&(J_{t,z}(c,a)\land b\in D(c+t-k,a+z-k))\\
    \lor&(J_{k,t}(b,c)\land a\in D(b+k-z,c+t-z)).\qedhere
\end{align*}

\end{proof}

\begin{coro}\label{cor:IntergerDef}
Let $L>4$ and $S\subseteq \Ss{L}$ be a dense subset. Then any g.e.c.\ graph $G$ on vertex set $S$ defines (with parameters) the ternary relations $\{C_{z,t,k} : z,t,k\in \ZZ\}$ and the partial functions $\{f_{z}:=(a\mapsto a + z) : z\in\ZZ\}$.
\end{coro}
\begin{proof}
If $L>4$, then $B(-,-)$ is definable from $G$ by \cref{B1Definable}. Let $a_0,\cdots,a_{n_L}=a_0$ be an orienting loop  in $S$. Then by \cref{prop:circularorder} we get that $f_z$ is definable for all $z\in\ZZ$.
\end{proof}
\begin{remark}\label{rem:why integer-distance free}
  \cref{cor:IntergerDef} explains why having pairs of elements of integer-distance is an obstacle to g.e.c.\ graphs in $\Ss{L}$ being isomorphic or even elementary equivalent. For example, suppose that $L>4$ is not an integer, and suppose that $S$ is a dense set which has exactly one pair $(a,b)$ with $d_L(a,b)<1$ and $d_L(a,b)=N \pmod{L}$ for some $N\in\NN$. Then $\{a,b\}$ is definable from any orienting loop in $S$. But by definition, some g.e.c.\ graphs on $S$ will have an edge between $a$ and $b$ but some will not: the theory of the structure $(S,E,a_0,\ldots a_{n-1})$ is not the same for any choice of g.e.c.\ graph on $S$, even when we fix the graph structure on $a_0,\ldots, a_{n-1}$.

  On the other hand, when $S$ is integer-distance free and $E_0,E_1$ are two sets of edges such that $(S,E_0)$ and $(S,E_1)$ are g.e.c.\ then for any choice of elements $a_0,\ldots,a_{n-1}$ from $S$, if $E_0$ and $E_1$ agree on $a_0,\ldots, a_{n-1}$, then $(S,E_0,a_0,\ldots,a_{n-1}) \equiv (S,E_1,a_0,\ldots,a_{n-1})$ by quantifier elimination, see \cref{thm:Elementary.Equiv.for.Circle.Graphs}.

  \begin{question}
    In \cref{rem:why integer-distance free} we named the constants for an orienting loop, is this necessary? Namely, is it true that in the situation described in the first part \cref{rem:why integer-distance free} there are sets of edges $E_0,E_1$ on $S$ such that $(S,E_0) \not \equiv (S, E_1)$ while both are g.e.c.?
  \end{question}
  
\end{remark}

\section{Elementary equivalence}\label{sec:elentaryEquivalence}

Our aim of this section is to generalize the elementary equivalence result from circles (\cref{thm:Elementary.Equiv.for.Circle.Graphs}) to an arbitrary separable metric space $(X,d)$ without isolated points. 
By \cref{rem:why integer-distance free} we know that the density of the vertex set alone does not guarantee all g.e.c.\ graphs in $(X,d)$ are elementary equivalent: in the case of circles, we need to assume additionally that the vertex set is integer-distance free. On the other hand, being dense and integer-distance free is a generic property of a countable subset of a circle: if we choose a countable set independently uniformly at random in $\Ss{L}$, almost surely it will be dense and integer-distance free. We might expect all g.e.c.\ graphs on ``generic'' countable sets are elementary equivalent. For this we need to determine when a set is generic. We will show in this section that a good notion of genericity is to require that every finite tuple in the countable set avoids any ``thin" definable set where thin means containing no non-empty open balls. If $(X,d)$ comes with a strictly positive probability measure (i.e., one that assigns a positive measure to every nonempty open ball) such that all thin definable sets are of measure 0, then an i.i.d. sampling of an infinite sequence of points in the metric space will almost surely be generic.

\begin{defn}[Defining structure and independent sets]\label{def:defining structure and ind set}
Let $(X,d)$ be a separable metric space. 
\begin{itemize}
    \item 
We call an $\mathcal{L}$-structure $ M=(M,\ldots)$ a \emph{defining structure} of $(X,d)$ if $X\subseteq M^n$ is an $\mathcal{L}$-definable set for some $n\in\NN$ and the binary relation $d(x,y) < 1$ is $\mathcal{L}$-definable over $\emptyset$.

\item We call a definable subset $A\subseteq X$ \emph{thick} in $(X,d)$ if $A$ contains a non-empty open ball. Otherwise, we say $A$  is \emph{thin}. 

\item We say that \emph{thickness is definable} in $ M$, if for any $\mathcal{L}$-formula $\varphi(x,y)$ without parameters and with $|x|=n$ and $|y|=m$, there is a formula $\psi(y)$ without parameters 
such that for any $\bar{b}\in X^m$, $M\vDash \psi(\bar{b})$ iff $\varphi(X,\bar{b})$ is thick. 

\item A subset $Y\subseteq X$ is called \emph{independent} if for all $\mathcal{L}$-formula $\varphi(x,y)$ with $|x|=n$ and $|y|=mn$ for some $m\in\NN$, for all $\bar{b}\in Y^m$, if $\varphi(X,\bar{b})$ is thin then $(Y\cap\varphi(X,\bar{b})) \subseteq \bar{b}$.


\end{itemize}
\end{defn}



\begin{example}
\begin{enumerate}    
    \item Being an independent subset of a metric space depends on the choice of the defining structure. Namely, given $(X,d)$, there can be an $\mathcal{L}$-structure $M$ and an $\mathcal{L}'$-structure $ M'$, which are both defining structures of $(X,d)$, and a dense subset $S$ which is independent in $ M$ but not independent in $ M'$. For example in the circle $\mathbb{S}_L$ for $L$ large enough, let $ M:=(\mathbb{S}_L,+_L,B(-,-), \bar{a})$ and $ M':=(\mathbb{S}_L,B(-,-), \bar{a})$, where $B(a,b)$ iff $d(a,b)<1$ and $\bar{a} = (a_0,\ldots, a_{n_L})$ is an orienting loop. In both structures, thick sets are those sets containing an interval, hence thickness is definable since the circular order is definable by \cref{prop:circularorder}. However, a dense set $S\subseteq \mathbb{S}_L$ with three distinct points $x_0,x_1,x_2$ such that $d(x_0,x_1)=d(x_1,x_2)\not\in\{z\pmod{L} : z\in\ZZ\}$ is not independent in $ M$ (because the set defined by $x-_L x_1 = x_1-_L x_2$ is thin). But $S$ can be independent in $ M'$ (the reason that $z$ is not an integer modulo $L$ is because the partial function $x \mapsto x+z$ is definable by \cref{prop:circularorder}).

    \item With a similar proof to that of \cref{thm:Elementary.Equiv.for.Circle.Graphs}, one can see that $\Th(M')$ has quantifier elimination after adding the circular order relation and the functions $x\mapsto x+z \pmod{L}$. From this, it is not hard to see that if $L$ is not an integer, 
    then a set $S\subseteq \mathbb{S}_L$ is dense independent iff it is dense and integer-distance free.
\end{enumerate}
\end{example}

Let $(X,d)$ be a separable metric space without isolated points and $M$ a defining structure in the language $\mathcal{L}$. Given a graph $G=(V,E)$ with $V\subseteq X$, we put on $V$ two structures, $V_{\text{ind}}$ and $G_{\text{ind}}$ in the languages $\mathcal{L}_{\text{ind}}$ and $\mathcal{L}_{\text{ind}}^E:=\mathcal{L}_{\text{ind}}\cup \{E\}$ respectively, where \[\mathcal{L}_{\text{ind}}= \{ R_{\phi} : \phi(x_1,\ldots,x_n) \in \mathcal{L}, x_i \text{ variables in the sort of }X \}\] and $R_{\phi}$ is an $n$-any relation interpreted as $\{(a_1,\ldots,a_n) \in V^n: M \vDash \phi(a_1,\ldots,a_n)\}$ in both structures.

\begin{thm}\label{thm:elementary equivalence for dense independent sets}
Let $(X,d)$ be a separable metric space without isolated points and $M$ a defining structure in the language $\mathcal{L}$. Let $V_1$ and $V_2$ be two independent dense sets of $X$. Suppose \emph{thickness is definable} in $ M$ 
and $G_1=(V_1,E_1)$, $G_2=(V_2,E_2)$ are g.e.c.\ graphs in $(X,d)$. Then $(G_1)_{\text{ind}}$ and $(G_2)_{\text{ind}}$ are elementary equivalent in $\mathcal{L}_{\text{ind}}^{E}$ (and hence also in the graph language $\mathcal{L}_G$ and in $\mathcal{L}_{\text{ind}}$).

Moreover, both $(V_1)_{\text{ind}}$ and $(G_1)_{\text{ind}}$ have quantifier elimination (in their respective languages). 

\end{thm}
\begin{proof}
We may assume $\mathcal{L}$ is countable, as it is enough to check elementary equivalence and quantifier elimination on all countable sublanguages. 
Take a non-principal ultrafilter $\mathcal{U}$ on $\NN$ such that the ultrapowers $( M^{*},(G_1)^{*}_{\text{ind}})=( M^{*},(V_1^{*},\ldots))$ and $( M^{*},(G_2)^{*}_{\text{ind}})=( M^{*},(V_2^{*},\ldots))$ of $( M,(G_1)_{\text{ind}})$ and $( M,(G_2)_{\text{ind}})$ along $\mathcal{U}$ are countably saturated, namely any consistent countable collection of formulas has a realization in $\mathcal{U}$. Clearly it is enough to show that $$(G_1)_{\text{ind}}^{*}\equiv_{\mathcal{L}_{\text{ind}}^E}(G_2)^{*}_{\text{ind}}.$$ 


We call a map $f:A\to B$ with $A\subseteq V_1^{*}$ and $B\subseteq V_2^{*}$ an \emph{$\mathcal{L}^E_{\text{ind}}$-embedding} if it satisfies:
\begin{itemize}
\item
$f$ is a graph isomorphism;
\item
Let $\bar{a}$ be any finite tuple of elements in $A$ and $\bar{b}=f(\bar{a})$. Then $\tp_{\mathcal{L}}(\bar{a})=\tp_{\mathcal{L}}(\bar{b})$ in $ M^{*}$.
\end{itemize}
(This is the same as an embedding in the relational language $\mathcal{L}^E_{\text{ind}}$.)

\begin{claim} \label{cla:back-n-forth}
Suppose $f:A\to B$ is an $\mathcal{L}^E_{\text{ind}}$-embedding, where  $A\subseteq V_1^{*}$ and $B\subseteq V_2^{*}$ are finite sets. Then for any $v_1\in V_1^{*}$, there is $v_2\in V_2^{*}$ such that $f\cup\{(v_1,v_2)\}$ is an $\mathcal{L}^E_{\text{ind}}$-embedding extending $f$. 
\end{claim}

Note that By symmetry it will follow that for any $v_2\in V_2^{*}$, there is $v_1\in V_1^{*}$ such that $f\cup\{(v_1,v_2)\}$ is an $\mathcal{L}^E_{\text{ind}}$-embedding.
\begin{proof}
We may assume that $v_1 \notin A$. Let $\bar{a}$ be an enumeration of $A$ and $\bar{b}$ be the image of $\bar{a}$ under $f$ which is an enumeration of $B$. By assumption $\text{tp}_{\mathcal{L}}(\bar{a})=\text{tp}_{\mathcal{L}}(\bar{b})$ in $ M^{*}$. Given $v_1\in V_1^{*}$. Let $p(x,\bar{y}):=\text{tp}_{\mathcal{L}}(v_1,\bar{a})$. Since $\bar{a}$ and $\bar{b}$ have the same $\mathcal{L}$-type, $p(x,\bar{b})$ is a consistent $\mathcal{L}$-type. Let $A_c:=\{a\in A: v_1\mathrel{E}^{*}_1a\}$. We only need to show that $$\Delta(x):=p(x,\bar{b})\cup\{x\mathrel{E}f(a):a\in A_c\}\cup\{\neg x\mathrel{E}b:b\in B\setminus f(A_c)\}$$ has a realization $v_2$ in $V_2^{*}$. Since $( M^{*},G_2^{*})$ is countably saturated, we only need to show that any finite subset of $\Delta(x)$ has a realization in $V_2^{*}$. Fix a formula  $\varphi(x,\bar{b})\in p(x,\bar{b})$. Choose for any $b\in B$ a representative $(b^{i})_{i\in \NN}$ in $M^{\NN}$. It suffices to show that $$\varphi(x,\bar{b}^{i})\land\bigwedge_{a\in A_c}x\mathrel{E}f(a)^{i}\land\bigwedge_{b\in B\setminus f(A_c)}\neg x\mathrel{E}b^{i}$$ has a solution in $( M,G_2)$ for $\mathcal{U}$-many $i\in \NN$.

Let $D:=\{a\in A:d(v_1,a)<1\}$. Note that $A_c\subseteq D$. Let

$$\psi(x,\bar{a}):= \varphi(x,\bar{a})\land \bigwedge_{a\in D} d(x,a)<1\, \land\, \bigwedge_{a'\in A\setminus D} d(x,a')\geq1.$$ 
Then $ M^{*}\vDash\psi(v_1,\bar{a})$. 


The fact that $V_1$ is dense independent and $v_1^i \in \psi(x,\bar{a}^i)\setminus \bar{a}^i$ for $\mathcal{U}$-many $i$, implies that $\psi( M,\bar{a}^i)$ is thick for all such $i \in \NN$.
Let $\xi(\bar{y})$ be such that for every $\bar{d}$, $M\vDash\xi(\bar{d})$ iff $\psi(M,\bar{d})$ is thick.
We therefore see that $ M^* \vDash \xi(\bar{a})$. 
Since $\text{tp}_{\mathcal{L}}(\bar{a})=\text{tp}_{\mathcal{L}}(\bar{b})$ we must also have that $ M^{*}\vDash \xi(\bar{b})$, therefore $\psi(M,\bar{b}^{i})$ is thick for $\mathcal{U}$-many $i\in \NN$. 
Hence for $\mathcal{U}$-many $i \in \NN$, $\psi(M,\bar{b}^{i})$ contains a non-empty open ball $B_{\epsilon_i}(v^i)$ for some $v^i\in X$. Let $u^i\in B_{\epsilon_i}(v^i)\cap V_2$. Then for $\mathcal{U}$-many $i \in \NN$ there is $\delta^i>0$ such that $B_{\delta^i}(u^i)\subseteq B_{\epsilon_i}(v^i)\subseteq \psi(M,\bar{b}^{i})$, $f(D)^{i}=\{f(a)^{i}:a\in D\}\subseteq  B_1(u^i)$ and $f(A_c)^{i}:=\{f(a)^{i}:a\in A_c\}\subseteq f(D)^{i}$. By g.e.c.\ there is $v_2^{i}\in V_2\cap B_{\delta^i}(u^i)$ such that $v_2^{i}\mathrel{E} d$ for all $d\in f(A_c)^{i}$ and $\neg v_2^{i}\mathrel{E} d'$ for all $d'\in f(D)^{i}\setminus f(A_c)^{i}$. As $v_2^{i}\in B_{\delta^i}(u^i)\subseteq \psi(M,\bar{b}^{i})$, we have $d(v_2,e)\geq 1$ for all $e\in B^{i}\setminus f(D)^{i}$. We conclude that \[(M,G_2)\vDash \varphi(v_2^{i},\bar{b}^{i})\land \bigwedge_{a\in A_c} v_2^{i}\mathrel{E}f(a)^{i}\,\land\,\bigwedge_{b\in B\setminus f(A_c)} \neg v_2^{i}\mathrel{E} b^{i}. \qedhere\]
\end{proof}

By \cref{cla:back-n-forth}, it is clear the Duplicator has a winning strategy in the Ehrenfeucht–Fra\"{i}ss\'{e} game between $(G_1)_{\text{ind}}^*$ and $(G_2)_{\text{ind}}^*$ (i.e., ensuring that in each round of the game the map remains an $\mathcal{L}_{\text{ind}}^E$-embedding) and thus $(G_1)_{\text{ind}}^{*}\equiv_{\mathcal{L}_{\text{ind}}^E}(G_2)_{\text{ind}}^{*}$ as required. The moreover part for $\mathcal{L}_{\text{ind}}^E$ follows from \cref{cor:qf-criteria}. Note that the same argument works with $\mathcal{L}_{\text{ind}}$ by proving an analogous claim for $\mathcal{L}_{\text{ind}}$-embedding so that the same is true for $\mathcal{L}_{\text{ind}}$.
\end{proof}
\begin{remark}
In model theory, there are well-studied expansions or reducts of existing structures which preserve nice properties. For example, lovely pairs, $H$-structures and so on (see \cite{BV}). Taking a dense independent set is similar to $H$-structures and the proof of the above Theorem is similar to \cite[Corollary 2.13]{BV}.
\end{remark}


\subsection{i.i.d.\ sets are dense independent}\label{sec-iid}
In this subsection, we show that the requirement for underlying sets being dense and independent for g.e.c.\ graphs are reasonable in the sense that when there is a probability measure compatible with the metric space $(X,d)$, then almost surely, any i.i.d.\ sampling of an infinite sequence of elements of $X$ will be dense and independent.

Throughout this section, we fix a compact separable metric space $(X,d)$ without isolated points and let the $\mathcal{L}$-structure $ M=(M,\ldots)$ be a defining structure for $(X,d)$.
\begin{defn}
We call a ($\sigma$-additive) probability measure $\mu$ on $X$ \emph{compatible with $ M$}, if the product measure $\mu^n$ measures all $\mathcal{L}$-definable subsets of $X^n$ for all $n\in\NN$, and for every definable set $D \subseteq X$, $\mu(D) =0$ iff $D$ does not contain any open ball (i.e., $D$ is thin).
\end{defn}


\begin{claim}\label{cla:thin sets get measure zero}
Let $\mu$ be a probability measure on $X$ compatible with $ M$.
Suppose $\phi(x,\bar{y})$ is a formula so that for every $\bar{b} \in X^m$ the set $D = \phi(X,\bar{b})$ is thin. 
Then $\mu^{m+1}(\{(a,\bar{b}) \in X^{m+1} : M \vDash \phi(a,\bar{b}) \}) = 0$.
\end{claim}
\begin{proof}
Let $A:=\{(a,\bar{b}) \in X^{m+1} : M \vDash \phi(a,\bar{b}) \}$ and $\mathbf{1}_A(x,\bar{y})$ be the characteristic function of $A$. Then by Fubini (which is applicable for products of probability measures), $$\mu^{m+1}(A)=\int_{X^{m+1}}\mathbf{1}_A(x,\bar{y})\,d\mu^{m+1}(x,\bar{y})=\int_{X^{m}}\left(\int\mathbf{1}_A(x,\bar{y})\,d\mu(x)\right)\,d\mu^{m}(\bar{y})=0.$$
\end{proof}

\begin{thm}\label{thm:almostsureindependent}
Suppose $\mu$ is a probability measure on $X$ compatible with $ M$ and such that $\mathcal{L}$ is countable and thickness is definable in $ M$. Then the set \[\{ (x_n)_n \in X^{\mathbb{N}} : \{x_n:n\in\NN\} \text{ is dense independent}\}\] is (measurable and) $\mu^{\mathbb{N}}$-conull.
\end{thm}
\begin{proof}
Let $\Ind:=\{ (x_n)_n \in X^{\mathbb{N}} : \{x_n:n\in\NN\} \text{ is independent}\}$ and $\mbox{Dense}:=\{ (x_n)_n \in X^{\mathbb{N}} : \{x_n:n\in\NN\} \text{ is dense}\}$ It suffices to show that $\mu^\NN(\Ind)=\mu^{\NN}(\mbox{Dense})=1$.

Let $(D_i)_{i\in\NN}$ be a countable basis of balls in $X$. By our assumptions on $\mu$ we have $\mu(D_i)>0$. Then $\{x_n:n\in \NN\}$ being dense is equivalent to ``for all $D_i$ there is some $x_{n_i}\in D_i$''. Hence it is Borel and
\begin{align*}
 \mu^{\NN}(\mbox{Dense})&\geq 1-\sum_{i\in\NN}\mu^{\NN}(\{(x_n)_n\in X^\NN: x_n\not\in D_i, \text{ for all }n\in\NN\})\\
 &=1-\sum_{i\in\NN}\left(\lim_{n\to\infty}(1-\mu(D_i))^n\right)=1.
\end{align*}

Let $\Delta:=\{\phi(x,\bar{y}): \text{for all }\bar{b}\in X^{m}, \phi(X,\bar{b})\text{ is thin}\}$. Since thickness is definable in $ M$, for any $\varphi(x,\bar{y})$, we may define $\phi(x,\bar{y}):=\varphi(x,\bar{y})\land\neg\psi(\bar{y})$ where $\psi(\bar{y})$ holds iff $\varphi(M^n,\bar{y})$ is thick; it follows that $\phi\in \Delta$. As $X$ has no isolated points, the formula $x=y$ is 
 in $\Delta$. Hence, to say that $\{x_n:n\in \NN\}$ is not independent is equivalent to saying that for some $\phi(x,\bar{y})\in \Delta$, $\phi(x_{i_0},\ldots, x_{i_m})$ holds. As $\mu$ is compatible with $M$ and $\mathcal{L}$ is countable, this set is Borel. Thus, \begin{align*}
&\mu^{\NN}(\Ind)=1-\mu^{\NN}(
    X^{\NN}\setminus\Ind)\\
    &\geq 1-\mu^{\NN}(\bigcup_{ \phi(x,\bar{y})\in\Delta, (i_0,\ldots,i_m)\in \NN^{m+1}}\{(x_n)_n \in X^{\mathbb{N}}:M\vDash\phi(x_{i_0},x_{i_1},\ldots,x_{i_m})\})\\
    &\geq 1-\sum_{\phi(x,\bar{y})\in\Delta, (i_0,\ldots,i_m)\in \NN^{m+1}}\mu^{m+1}(\{(a,\bar{b}) \in X^{m+1} : M \vDash \phi(a,\bar{b})\})\\
    &=1,
\end{align*} where the last equality is by \cref{cla:thin sets get measure zero}.
\end{proof}

With the same assumptions as above, let $(x_n)_{n \in \NN}$ 
be a sequence of points from $X$ be chosen independently uniformly according to $\mu$ at random. 

By the above, almost surely $V:=\{x_n : n\in \NN \}$ is an independent dense set. 
We make $V$ into a graph $\mathcal{G}_{V,p}$ with vertex set $V$ and put an edge between $x_i$ and $x_j$ for $d(x_i,x_j)<1$ with probability $0<p<1$ independently. Then by \cite[Lemma 2.1]{BJ}, almost surely $\mathcal{G}_{V,p}$ is a g.e.c. graph. Hence by \cref{thm:elementary equivalence for dense independent sets}, there is a complete first-order theory $T$ such that almost surely $\mathcal{G}_{V,p} \vDash T$. 

\begin{remark} \label{rem:Alex2}
    As we mentioned in \cref{rem:Alex Kruckman}, a variant of this result has already been obtained in \cite[Theorem 1.4.8, Example 2.1.1]{KruckmanThesis} and \cite[Example 3.1]{ackerman2017properly}. On one hand, there, the elementary equivalence was proved for any countable fragment in the infinitary logic $L_{\omega_1,\omega}$ and they do not require existence of defining structures with which the measure is compatible. On the other hand, the result there depends on the choice of measure and a priori gives a different theory for any choice of measure (which is not the case in our result, as long as the measure is compatible with the defining structure). Additionally, \cref{thm:elementary equivalence for dense independent sets} can be applied to metric spaces without a natural measure associated to it. For example, in \cref{sec:Urysohn}, we will discuss geometric random graphs on the Urysohn space, where it is not clear how to construct a probability measure which assigns positive measures to non-empty balls.
\end{remark}

\subsection{Riemmanian submanifold of \texorpdfstring{$\mathbb{R}^n$}{Rn}}
In this subsection, we will discuss a class of metric spaces for which there is a natural defining structure where thickness is definable. The class contains Riemmanian submanifolds definable in some o-minimal expansion the real field, and in general any metric space definable in some o-minimal structure such that definable sets which contain non-empty open balls are exactly those of maximal o-minimal dimension. 

\begin{defn}[o-minimal structures]
Let $ M=(R,<,\cdots)$ be an expansion of a totally ordered structure. $M$ is called o-minimal if every definable subset of $ M$ is a finite union of points and open intervals whose endpoints lie in $ M\cup\{\pm\infty\}$. $M$ is called an \emph{o-minimal expansion of the real field} if in addition $M=(\RR,+,\cdot,0,1,<,\ldots)$ is an expansion of the field of real numbers.
\end{defn}

  \begin{fact}(see \cite{vandenDries})
  The following structures are o-minimal:
\begin{itemize}
    \item 
    $\bar{\RR}:=(\RR,+,\cdot,0,1,<)$.
    \item
    $\RR_{an}:$ the field of real numbers expanded by all restricted analytic functions on $[-1,1]^n, n\in \NN$;
    \item
    $\RR_{an,exp}:$ the expansion of $\RR_{an}$ by the function $x\mapsto \exp{(x)}$.
\end{itemize}
  \end{fact}

\begin{defn}[o-minimal dimension]
    Let $M=(R,<,\ldots)$ be an o-minimal structure and $X\subseteq R^n$ be a definable set. Define the \emph{o-minimal dimension of $X$}, $\dim(X)$, to be the maximal $m\leq n$ satisfying there exists a coordinate projection $\pi:R^n\to R^m$ such that $\pi(X)$ contains a non-empty open subset of $R^m$.
\end{defn}

\begin{fact}(\cite[Chapter 4, Proposition 1.5]{VDD}) \label{fac:o-minimal->dimension is definable}
Let $M$ be an o-minimal structure. Then the o-minimal dimension is definable in the following sense. For any formula $\phi(x;y)$ there are formulas $\{\psi_i(y),0\leq i\leq |x|\}$ such that for any $b\in M^{|y|}$, $\dim(\phi( M,b))=i$ iff $ M\vDash \psi_i(b)$. 
\end{fact}

From \cref{fac:o-minimal->dimension is definable} we immediately get:

\begin{prop}\label{prop:o-minimal}
Let $(X,d)$ be a separable metric space without isolated points.
    Suppose that $M$ is an o-minimal structure such that $X \subseteq M^n$ is a definable set and the relation $d(x,y)<1$ is definable in $M$. Suppose that $\dim(X)=k$ and that a definable subset $Y\subseteq X$ is thick iff $\dim(Y)=k$. Then $M$ is a defining structure for $(X,d)$ in which thickness is definable. 
\end{prop}

Our chief examples are Riemannian submanifodls\footnote{See \cite{Lee} for more on Riemannian manifolds.} of $\mathbb{R}^n$, e.g., spheres.

\begin{prop}\label{prop:O-minDefiningStructure}
Let $(X,g)$ be a Riemannian submanifold of $\RR^n$ of dimension $m$.
Assume that $X\subseteq \RR^n$ is a definable set in some o-minimal expansion $\mathcal{R}$ in a countable language of the real field (e.g., a countable reduct of $\RR_{an,exp}$),
which also defines the relation $d(x,y)<1$ for $x,y \in X$ where $d=d_g:X \times X \to X$ is the geodesic distance. 
Then thickness is definable in $\mathcal{R}$. And if $\mu$ is a probability measure equivalent to the $g$-induced volume measure $V_g$ on $X$\footnote{That is, $\mu(A) > 0$ iff $V_g(A)>0$ for every Borel subset $A \subseteq X$.}, then $\mu$ is compatible with $\mathcal{R}$.



\end{prop}
\begin{proof}
To show that that thickness is definable, by \cref{prop:o-minimal}, it is enough to show that for any definable $D \subseteq X$, $D$ is thick iff $\dim(D)=m$. We will use the fact that in $\mathcal{R}$, the o-minimal dimension coincides with the manifold dimension for any definable submanifold of $\RR^n$. Suppose that $D$ is thick. Then it contains an open ball in $X$. This open ball $U$ is itself a submanfiold of $\RR^n$ of dimension $m$, and hence $m=\dim(U)\leq\dim(D)\leq\dim(X)=m$ as required. On the other hand, if $\dim(D)=m$, By \cite[Chapter 4, Cororllary 1.9]{VDD}, $\dim(D) = \dim(X)$ implies that $D$ contains a nonempty open subset in $X$, and thus thick.

It remains to show that $\mu$ is compatible with $\mathcal{R}$. By o-minimality (i.e., cell-decomposition, see \cite[Chapter 3]{vandenDries}), $\mathcal{R}$-definable sets are finite union of locally closed sets, thus Borel. In particular, definable subsets of $X^k$ are $\mu^k$-measurable. We show the second part of the definition, i.e., that $\mu(D)>0$ iff $D$ is thick. For the implication from right to left note that any open set gets positive measure (since this is true for the volume measure). On the other hand, assume that $D$ is not thick. As $\dim(D) = \dim(\cl (D))$, we may assume that $D$ is closed. As $X$ is definable it has a finite atlas, and in each part the volume measure of $D$ is 0 (since it is a closed set of smaller dimension), and hence the measure of $D$ is 0.
\end{proof}

\section{How to detect volumes}\label{sec:Volumn}
As shown by \cite{AS}, in the case of circle $\Ss{L}$, the length $L$ is an information that can be recovered from any g.e.c.\ graph on a dense subset of $\Ss{L}$. It is not hard to see that this information is coded in the first-order theory $T_{\mathcal{L}_G}(\Ss{L})$ of these graphs. In this section we present a generalization of this result.

The key idea in \cite{AS} is to associate a real number $\alpha(G)$ to any graph and interpret the meaning of $\alpha(G)$ in terms of the underlying metric space $\Ss{L}$ when $G$ is a geometric random graph in $\Ss{L}$. We will follow their approach.

\begin{defn}
Let $G=(V,E)$ be a graph, define $$\alpha(G):=\inf_{U\subseteq_{\text{fin}} V }\sup_{v\in V}\frac{|N_1(v)\cap U|}{|U|} \in [0,1],$$ where $U\subseteq_{\text{fin}} V$ means $U$ is a finite subset of $V$.
\end{defn}

We first show that $\alpha(G)$ is coded in the theory of $G$.

\begin{lem}
There is a countable set of sentences $\{\phi_{m,n}:m\leq n\}$ in the graph language $\mathcal{L}_G$ such that the following holds for any graph $G$: $$\alpha(G)=\inf\{m/n: G\vDash \phi_{m,n}\}.$$ Therefore, if two graphs have different values for $\alpha(G)$, then they cannot have the same first-order theory. 
\end{lem}
\begin{proof}
Let $\phi_{m,n}$ be the sentence expressing that there is a set $U$ of size $n$ such that for any $v$ in the vertex set $V$, $|N_1(v)\cap U|\leq m$. By definition, if $G\vDash \phi_{m,n}$ then $\alpha(G)\leq m/n$, hence $\alpha(G)\leq \inf\{m/n: G\vDash \phi_{m,n}\}=:r$. We need to show $\alpha(G)=r$. Suppose towards a contradiction that $\alpha(G)<r$, then there is a finite $U_0\subseteq V$ of size $n_0$ for some $n_0>0$, such that $|N_1(v)\cap U_0|< rn_0$ for all $v\in V$. Let $m_0:=\lceil rn_0\rceil-1$, then $m_0<rn_0$ and $|N_1(v)\cap U_0|\leq m_0$ for all $v\in V$. Thus, $G\vDash \phi_{m_0,n_0}$ and $m_0/n_0<r=\inf\{m/n: G\vDash \phi_{m,n}\}$, a contradiction.
\end{proof}


\begin{defn}
Let $(X,d)$ be a metric space. We call $(X,d)$ \emph{uniform} if there is a \emph{uniformly distributed measure} $\mu$ on $(X,d)$, i.e., $\mu$ satisfies $0<\mu(B_r(x))=\mu(B_r(y))<\infty$ for all $r\in\mathbb{R}^{>0}$ and $x,y\in X$.
Suppose $(X,d)$ uniform with $\mu$ a uniformly distributed measure, define \[Vl_{\mu}(X):=\mu(X)/\mu(B_1(x))\in \mathbb{R}^{\geq 0}\cup\{\infty\}\] for some $x\in X$, and call it the \emph{ball volume} of $X$.
\end{defn}
\begin{remark}
 By the Christensen's lemma \cite{Ch70}, if the metric space $(X,d)$ is separable, then uniformly distributed Borel regular measures are unique up to scalar multiple. In this case $Vl_{\mu}(X)$ does not depend on $\mu$.   
\end{remark}

Examples of uniform metric spaces are $\mathbb{R}^{d}$ with the euclidean metric and the $n$-sphere $S^{n}$ with geodesic metric (witnessed by the spherical measure). Other examples include locally compact groups with a left-invariant metric (e.g., the torus) in which case the Haar measure is uniformly distributed. A non-example is the disc in $\mathbb{R}^{2}$ (see \cite[Proposition 2.4]{KP}).

We will further use the VC-theorem to approximate the uniformly distributed measure by finite sets. 

\begin{defn}[VC-dimension]
Let $X$ be a set and $\mathcal{F}\subseteq\mathcal{P}(X)$. We say that $A\subseteq X$ is shattered by $\mathcal{F}$ if for every $S\subseteq A$ there is $F\in\mathcal{F}$ such that $F\cap A=S$. The VC-dimension of $\mathcal{F}$ is the smallest integer $n$ such
that no subset of $X$ of size $n+1$ is shattered by $\mathcal{F}$. If no such $n$ exists, then we say $\mathcal{F}$ has infinite VC-dimension.
\end{defn}
For the next fact, see \cite[Chapter 9, Theorem 1]{Empirical} and \cite[Theorem 6.6 \& the discussion above Example 6.7]{PierreBook}.
\begin{fact}[Corollary of the VC-theorem]\label{fact-VC}
Let $(X,\mu)$ be a probability space. Suppose $\mathcal{S}$ is a countable family of measurable subsets of $X$ such that $\mathcal{S}$ is of finite VC-dimension. Then for any $\epsilon>0$, for $n$ large enough, there are $(x_1,\cdots,x_n)\in X^{n}$ such that $$\sup_{S\in\mathcal{S}}\left|\frac{1}{n}\sum_{i\leq n}\mathbf{1}_{S}(x_i)-\mu(S)\right|<\epsilon,$$ where $\mathbf{1}_S:X\to\{0,1\}$ is the characteristic function of $S$.
\end{fact}

\begin{lem}\label{lem-VC}
Let $(X,d)$ be a separable metric space and $\mu$ a uniformly distributed measure on $X$. Let $V$ be a countable dense subset of $X$. Suppose $\mu$ satisfies  $$\lim_{\epsilon\in \mathbb{Q}\to 0}|\mu(B_1(x))-\mu(B_{1+\epsilon}(x))|=0 \text{ for all (some) } x \in X$$ and the family $\mathcal{S}:=\{B_{1+\epsilon}(v):v\in V,\epsilon\in(-1,1)\cap\mathbb{Q}\}$ is of finite VC-dimension. Then there is a sequence $(U_n)_{n\in\mathbb{N}}$ of finite subsets of $V$ such that for any sequence $(v_i)_{i\in\mathbb{N}}\in V$, we have $$\lim_{i\to\infty}\frac{|B_1(v_i)\cap U_i|}{|U_i|}=\mu(B_1(x))$$ for any $x\in X$.
\end{lem}
\begin{proof}
Fix an $\epsilon>0$ and $x\in X$. Let $\delta\in\mathbb{Q}^{>0}$ be such that $|\mu(B_1(x))-\mu(B_{1\pm\delta}(x))|<\epsilon$ for any $x\in X$. Let $(x_1,\ldots,x_n)\in X^{n}$ given by \cref{fact-VC}, i.e., $\sup_{S\in\mathcal{S}}\left|\frac{1}{n}\sum_{i\leq n}\mathbf{1}_{S}(x_i)-\mu(S)\right|<\epsilon.$ Let $U_{\epsilon}$ be a set of distinct $x_1',x_2',\ldots,x_n'\in V$ such that $d(x_i,x_i')<\delta$ for $i\leq n$ (this is possible by the density of $V$; note that it is possible that $x_i=x_j$ for some $i\neq j$, but still we require that $x_i' \neq x_j'$). Note that by the triangle inequality $d(v,x_i)-\delta\leq d(v,x_i')\leq d(v,x_i)+\delta$ for all $v\in V$. Therefore, for any $v\in V$, $$ \frac{1}{n}\sum_{i\leq n}\mathbf{1}_{B_{1-\delta}(v)}(x_i)\leq \frac{|B_1(v)\cap U_\epsilon|}{|U_\epsilon|}\leq \frac{1}{n}\sum_{i\leq n}\mathbf{1}_{B_{1+\delta}(v)}(x_i).$$ By assumption $\frac{1}{n}\sum_{i\leq n}\mathbf{1}_{B_{1+\delta}(v)}(x_i)<\mu(B_{1+\delta}(v))+\epsilon<\mu(B_{1}(v))+2\epsilon=\mu(B_{1}(x))+2\epsilon$. Similarly, $\frac{1}{n}\sum_{i\leq n}\mathbf{1}_{B_{1-\delta}(v)}(x_i)>\mu(B_{1-\delta}(v))-\epsilon>\mu(B_{1}(x))-2\epsilon$. Therefore $$\sup_{v\in V}\left|\frac{|B_1(v)\cap U_\epsilon|}{|U_\epsilon|}-\mu(B_1(x))\right|<2\epsilon.$$ Finally, let $\epsilon_n := 1/n$ and $U_n:=U_{\epsilon_n}$.
\end{proof}

\begin{thm}\label{thm-volume}
Suppose $(X,d)$ is a separable uniform metric space with a uniformly distributed measure $\mu$ such that $\mu(X)<\infty$. Let $G$ be a g.e.c.\ graph on a dense subset of $X$. Suppose $\mu$ satisfies $$\lim_{\epsilon\in \mathbb{Q}\to 0}|\mu(B_1(x))-\mu(B_{1+\epsilon}(x))|=0$$ for any (some) $x\in V$, and the family $\mathcal{S}:=\{B_{1+\epsilon}(v):v\in V,\epsilon\in(-1,1)\cap\mathbb{Q}\}$ is of finite VC-dimension. Then $$\alpha(G)=1/Vl_{\mu}(X).$$
\end{thm}
\begin{proof} Let $b:=\mu(B_1(X))$.
We may normalize $\mu$ and assume that $\mu(X)=1$, so that by definition we need to show that $\alpha(G) = b$. We first show that $\alpha(G)\leq b$. Let $\{U_n:n\in\mathbb{N}\}$ be a family of finite sets in $V$ given by \cref{lem-VC} (after normalising the measure of $X$ to be 1), i.e., for any sequence $(v_i)_{i\in\mathbb{N}}\in V$, we have $$\lim_{i\to\infty}\frac{|B_1(v_i)\cap U_i|}{|U_i|}=b.$$
Therefore, $$\alpha(G) = \inf_{U\subseteq_{\text{fin}} V }\sup_{v\in V}\frac{|N_1(v)\cap U|}{|U|} \leq \lowlim_{n\to\infty}\sup_{v\in V}\frac{|N_1(v)\cap U_n|}{|U_n|}\leq \lowlim_{n\to\infty}\sup_{v\in V}\frac{|B_1(v)\cap U_n|}{|U_n|}=b.$$ The last equality is because $\sup_{v\in V}\frac{|B_1(v)\cap U_n|}{|U_n|}=\max_{v\in V}\frac{|B_1(v)\cap U_n|}{|U_n|}=\frac{|B_1(v'_n)\cap U_n|}{|U_n|}$ for some $v_n'\in V$ and $$\lowlim_{n\to\infty}\sup_{v\in V}\frac{|B_1(v)\cap U_n|}{|U_n|}=\lowlim_{n\to\infty}\frac{|B_1(v'_n)\cap U_n|}{|U_n|}=\lim_{n\to\infty}\frac{|B_1(v'_n)\cap U_n|}{|U_n|}=b.$$ Together, $\alpha(G)\leq b$.

Now we show that $\alpha(G)\geq b$. Let $U$ be an arbitrary finite subset of $V$. Then
 \begin{align*}
\sup_{x\in X}\frac{|B_1(x)\cap U|}{|U|}\geq \int \frac{|B_1(x)\cap U|}{|U|}\,d\mu(x)\\
=\frac{1}{|U|}\sum_{u\in U}\int \mathbf{1}_{u\in B_1(x)}\,d\mu(x)\\
=\frac{1}{|U|}\sum_{u\in U}\int \mathbf{1}_{x\in B_1(u)}\,d\mu(x)\\
=\frac{1}{|U|}\sum_{u\in U}\mu(B_1(u))=b.
\end{align*}
As $\sup_{x\in X}\frac{|B_1(x)\cap U|}{|U|}=\max_{x\in X}\frac{|B_1(x)\cap U|}{|U|}=\frac{|B_1(x_0)\cap U|}{|U|}$ for some $x_0\in X$. By density and g.e.c.\ there is $v_0\in V$ such that $N_1(v_0)\cap U=B_1(x_0)\cap U$. Therefore, $\sup_{v\in V}\frac{|N_1(v)\cap U|}{|U|}\geq b$ for all finite $U\subseteq V$. Thus, $\alpha(G)\geq b$ as required.

\end{proof}
\begin{remark}
The hypothesis that the collection $\mathcal{S}=\{B_{1+\epsilon}(v):v\in V,\epsilon\in(-1,1)\cap\mathbb{Q}\}$ has finite VC-dimension holds in many cases, including all metric spaces which are subspaces of those definable in an o-minimal expansions of the real field, for example subspaces of $\RR^n$ with the induced metric, or the sphere and the torus. 
\end{remark}

We obtain the following corollary which is a generalization of Theorem 1.1 in \cite{AS}.
\begin{coro}\label{cor:GTh.recovers.L.in,S_L}
The graph theory $T_{\mathcal{L}_G}(\Ss{L})$ of a circle $\Ss{L}$, $L>2$ recovers $L$. In particular, if $L_1,L_2 > 2$ are distinct then $T_{\mathcal{L}_G}(\Ss{L_1}) \neq T_{\mathcal{L}_G}(\Ss{L_2})$. Similarly the graph theory $T_{\mathcal{L}_G}(S_r)$ of a sphere recovers the diameter $r$ for $r>1/\pi$.
\end{coro}
\begin{proof}
Note that both circles and spheres are definable in the reals, and balls of different radii are uniformly definable. Therefore, the family of balls is of finite VC-dimension. By \cref{thm-volume}, $\alpha(G)=1/Vl_{\mu}(\Ss{L})=2/L$ in the case of $\Ss{L}$ and in the case of $S_r$, by the area formula of a spherical cap, we get \[\alpha(G)=1/Vl_{\mu}(S_r)=2\pi r^2(1-\cos{(1/r)})/4\pi r^2=(1-\cos{(1/r)})/2.\] 
\end{proof}

\section{Dimension and Burden}\label{sec:Burden}
Given a metric space $(X,d)$, suppose it is definable in some o-minimal expansion of the real field. Then there is a natural notion of dimension associated to $X$ which is the o-minimal dimension, which equals to the topological dimension and the Hausdorff dimension (see \cite{FHW17} for more details). In this section, we explore the question that whether the first-order theory of g.e.c.\ graphs on dense independent subsets of $(X,d)$ recovers the dimension of $(X,d)$. In model theory, a ``tame'' theory can be associated with a well-behaved notion of rank, which is analogous to dimension in geometries, for example Morley-rank for $\omega$-stable theories, $SU$-rank for simple theories and so on. Here, we use \emph{burden} in NTP$_{2}$ theories. We will show that the burden of the theory of g.e.c.\ graphs in $(X,d)$ will be bounded above by the dimension of $X$. In particular, these theories are NPT$_2$ with finite burden.

Let us first introduce the notions of burden and dp-rank. They can be both defined by some patterns of formulas, or equivalently by certain configurations of indiscernible sequences. 
\begin{defn}[indiscernible sequences]
An infinite sequence $I=(a_i)_{i<\omega}$ of tuples is said to be \emph{indiscernible (quantifier-free indiscernible respectively)} over a set of parameters $A$ (or $A$-indiscernible), if $$\tp(a_{i_0},\ldots,a_{i_n}/A)=\tp(a_{j_0},\ldots,a_{j_n}/A)$$ ($\qftp(a_{i_0},\ldots,a_{i_n}/A)=\qftp(a_{j_0},\ldots,a_{j_n}/A)$ respectively) for any two sequences $i_0<i_1<\ldots<i_n$ and $j_0<j_1<\ldots<j_n$ and any $n\in\NN$.

A collection of indiscernible sequences $(I_\alpha)_{\alpha<\kappa}$ is said to be \emph{mutually indiscernible} over $A$, if $I_{\alpha}$ is indiscernible over $A\cup \bigcup_{\beta\neq\alpha} I_{\beta}$.

\end{defn}

\begin{defn}[dp-rank] \label{def:dp-rank} An \emph{ict-pattern} in a partial type $\Sigma(x)$ of depth $\kappa$ consists of a sequence of formulas $(\varphi_\alpha(x,y_\alpha))_{\alpha<\kappa}$ and an array of tuples $(a_{i,\alpha})_{i<\omega,\alpha<\kappa}$ such that for any function $\eta:\kappa \to \omega$ there is some $b_\eta \vDash \Sigma$ (namely, $\psi(b_\eta)$ holds for all $\psi\in\Sigma(x)$) such that $\varphi(b_\eta,a_{i,\alpha})$ holds iff $\eta(\alpha) = i$.  

The \emph{dp-rank} of $\Sigma(x)$, denoted as $\dpr(\Sigma)$, is the supremum of the depths of all ict-patterns in $\Sigma(x)$.
\end{defn}

\begin{fact}\label{rem:dp-rank equivalence} \cite[Proposition 4.22]{Simon14}
        Given a partial type $\Sigma(x)$ over $A$ and a cardinal $\kappa$, $\dpr(\Sigma)<\kappa$ iff for every sequence of $A$-mutually indiscernible infinite sequences $(I_{\alpha})_{\alpha<\kappa}$, and for every $b\vDash\Sigma$ there is $\beta<\kappa$ such that $I_\beta$ is indiscernible over $bA$. 
\end{fact}

\begin{fact}\label{fac:o-minimal dim = dpr} \cite[Theorem 3.8]{Simon14}
Suppose that $M$ is o-minimal and $X$ is definable, then $\dpr(X)$ is the o-minimal dimension of $X$.
\end{fact}

\begin{defn}[Burden]\label{def:burden} 
An \emph{inp-pattern} in a partial type $\Sigma(x)$ of depth $\kappa$ consists of a sequence of formulas $(\varphi_\alpha(x,y_\alpha))_{\alpha<\kappa}$ and an array of tuples $(a_{i,\alpha})_{i<\omega,\alpha<\kappa}$ such that 
\begin{itemize}
    \item 
    $\{\varphi_\alpha(x,a_{i,\alpha})\}_{i<\omega}$ is $k_\alpha$-inconsistent for each $\alpha<\kappa$ and for some natural number $k_\alpha$.
    \item
    $\{\varphi_\alpha(x,a_{\eta(\alpha),\alpha})\}_{\alpha<\kappa}\cup\Sigma(x)$ is consistent for any function $\eta:\kappa \to \omega$.
\end{itemize}
The \emph{burden} of $\Sigma(x)$, denoted as $\bdn(\Sigma)$, is the supremum of the depths of all inp-patterns in $\Sigma(x)$.

    A theory $T$ is NTP$_2$ if $\bdn(T):=\bdn(x=x)<\infty$ (here $x$ is a singleton), \textit{i.e.}, there is some cardinal $\kappa$ such that $\bdn(T)<\kappa$ (this is the same as saying that there is no inp-pattern with the same formula and the same $k$). If $\bdn(T)<2$ we  say that $T$ is \emph{inp-minimal}.
\end{defn}

\begin{fact} \label{rem:burden equivalence}\cite[Lemma 2.4]{ArtemNTP2}
  Given a partial type $\Sigma(x)$ over $A$ and a cardinal $\kappa$, $\bdn(\Sigma)<\kappa$ iff for every $A$-mutually indiscernible sequences $(I_{\alpha})_{\alpha<\kappa}$, with $I_{\alpha}=(a_{\alpha,i})_{i<\omega}$ and every $b\vDash\Sigma$ there is $\beta<\kappa$ such that there exists $I'$ indiscernible over $bA$ with $I'\equiv_{a_{\beta,0}A}I_{\beta}$.
\end{fact}

To prove that the burden is bounded by the o-minimal dimension, we need a technical lemma first.
\begin{lemma} \label{lem:burden technical lemma}
Let $M$ be an o-minimal expansion of the real field in the language $\mathcal{L}$.
Suppose $M$ is a defining structure of a separable metric space $(X,d)$  without isolated points (see \cref{def:defining structure and ind set}) such that thickness is definable. Suppose that $S \subseteq X$ is dense and independent, $G= (S,E)$ is a g.e.c.\ graph. Let 
\[(M^*,X^*,S^*,E^*,\mathbb{R}^*)\]
be a non-principal ultrapower of $(M,X,S,E,\mathbb{R})$. Suppose $a \in S^*$ is a singleton, and that $I = (a_0, a_1 ,\ldots ) \subseteq S^*$ is an $a$-indiscernible sequence of finite tuples in $M^*$ (i.e., in language $\mathcal{L}$) and quantifier-free indiscernible in $G^*$ (i.e., in the graph language $\mathcal{L}_G$), perhaps not over $a$. Then there is some $a'\in S^*$ such that $\tp_{\mathcal{L}^E_{\text{ind}}}(a'a_i) = \tp_{\mathcal{L}^E_{\text{ind}}}(a a_0)$ in $G^*$ for all $i<\omega$, where $\mathcal{L}^E_{\text{ind}}$ is defined as in \cref{thm:elementary equivalence for dense independent sets}. 
\end{lemma}

\begin{proof}
By  \cref{thm:elementary equivalence for dense independent sets}, it is enough to find $a'$ such that $\tp_{\mathcal{L}}(a'a_i) = \tp_{\mathcal{L}}(a a_0)$ in $M^*$ and $\qftp_{\mathcal{L}_G}(a' a_i) = \qftp_{\mathcal{L}_G}(a a_0)$ in $G^*$. 
We may assume that $a \notin I$. Let $r(x,y) = \qftp_{\mathcal{L}_G}(a a_0)$. Note that this is a finite set of formulas describing the edge relations between $a$ and elements from $a_0$ (because $a_0$ is a finite tuple). By compactness it is enough to prove that given any $\mathcal{L}$-formula $\theta(x,y)\in\tp_{\mathcal{L}}(a,a_0)$, $S(x) \land \bigwedge_{i<k} \theta(x,a_i) \land \bigwedge_{i<k} r(x,a_i)$ is consistent. We may assume that for each $j<|y|$, $\theta(x,y)$ implies $d(x,y_j)<1$ or $d(x,y_j) \geq 1$.

Now, $\theta(x,a_i)$ contains a (non-standard) open ball $B_{\epsilon}(c)$ for some $c\in M^*$ and positive $\epsilon\in \mathbb{R}^{*}$. Indeed, this follows since $S^*$ is an ultrapower of a dense independent set,  \[a\vDash\bigwedge_{i<k} \theta(x,a_i)\] and $\{a,a_0,\ldots,a_{k-1}\}\subseteq S^*$. As $\theta(x,a_i)$ implies $d(x,(a_i)_j)<1$ or $d(x,(a_i)_j)\geq 1$ for every $j<|y|$, we may apply g.e.c.\ to $c$ and $A:=\{(a_i)_j:i<k, j<|y|, (a_0)_j\mathrel{E^*}a\}$ and $B:=\{(a_i)_j:i<k, j<|y|, \neg ((a_0)_j\mathrel{E^*}a)\}\cap B_1(c)$ in $B_1(c)$,\footnote{Note that the g.e.c.\ property of $(S,E)$ transfers to the ultrapower.} and find $a'\in B_{\epsilon}(c)\cap S^*$ which is $E$-adjacent to all points in $A$ and none of the points in $B$ (note that $A \cap B = \emptyset$ by indiscernibility). Now, by the choice of $c$ and $\epsilon$, $a'$ satisfies $S(x) \land \bigwedge_{i<k} \theta(x,a_i) \land \bigwedge_{i<k} r(x,a_i)$ as required.
\end{proof}

\begin{thm} \label{thm:bound on burden}
Let $M$ be an o-minimal structure 
which is a defining structure of a separable metric space $(X,d)$ without isolated points. Suppose the o-minimal dimension of $X$ is $\ell$. Let $S\subseteq X$ be an independent dense set and $G=(S,E)$ be a g.e.c.\ graph in $(X,d)$, then $T_{\mathcal{L}_G}(X,d):=\Th(G)$ has burden at most $\ell$. In particular $T_{\mathcal{L}_G}(X,d)$ is NTP$_2$.
\end{thm}
\begin{proof}
We will prove $Th_{\mathcal{L}^E_{\text{ind}}}(G)$ has burden at most $\ell$.
Let 
\[(M^*,X^*,S^*,E^*,\mathbb{R}^*)\]
be a non-principal ultrapower of $(M,X,S,E,\mathbb{R})$. 
Let $a \in G^*$, and suppose that $(I_{k})_{k<\ell+1}$ is a family of mutually indiscernible sequences of tuples in $G^*$ . The indiscerniblity is in the sense of $\mathcal{L}^E_{\text{ind}}$ which means that they are mutually indiscernible in the sense of $\mathcal{L}$ and in the sense of the graph language $\mathcal{L}_G$. By \cref{rem:burden equivalence} and saturation, it is enough to show that for some $k<\ell+1$, there is some $I_k' \equiv_{aa^k_0} I_k$ where $I_k = (a^k_0, a^k_1 , \ldots )$ and the equivalence is in $\mathcal{L}^E_{\text{ind}}$ (in $G^*$). (The sequence $I_k'$ may be in an elementary extension of $G^*$, but note that any witness of large burden exists in $G^*$ by countable saturation.)

By \cref{fac:o-minimal dim = dpr}, $\dpr(X) \leq \ell$, and hence (by \cref{rem:dp-rank equivalence}) for some $k<\ell+1$, $I_k$ is $a$-indiscernible in the sense of $\mathcal{L}$. Note that $I_k$ is indiscernible in $G^*$ by assumption. By \cref{lem:burden technical lemma}, there is some $a' \in G^*$ such that $\tp_{\mathcal{L}^E_{\text{ind}}}(a' a^k_i) = \tp_{\mathcal{L}^E_{\text{ind}}}(a a^k_0)$ for all $i<\omega$ in $G^*$. By Ramsey and compactness there is $I_k' \equiv_{\mathcal{L}^E_{\text{ind}}} I_k$ which is indiscernible over $a'$ 
and have the same EM-type as $I_k$ over $a'$, namely all formulas $\varphi(x_1,\ldots,x_n,a')$ such that $\varphi(a^k_{i_1},\ldots,a^k_{i_n},a')$ holds for all $i_1<i_2<\ldots<i_n$. Write $I_k' = (a'_0, a'_1, \ldots)$. In particular, $\tp_{\mathcal{L}^E_{\text{ind}}}(a' a'_i) = \tp_{\mathcal{L}^E_{\text{ind}}}(a' a^k_i) =  \tp_{\mathcal{L}^E_{\text{ind}}}(a a^k_0)$ for all $i<\omega$. Working in an elementary extension, we may apply an automorphism sending $a'a_0'$ to $aa^k_0$, and it sends $I_k'$ to $I_k''=(a^k_0,a_1'',\ldots)$ starting with $a^k_0$ and is $a$-indiscernible. Therefore, $I_k'' \equiv_{aa^k_0} I_k$ in $\mathcal{L}^E_{\text{ind}}$ as desired.


\end{proof}

\begin{coro}
The theory of g.e.c.\ graphs on a dense independent set of the circle $\mathbb{S}_L$ of length $L>0$ is inp-minimal. 
\end{coro}
Note that when $L>4$, circular order is definable by Section 2, hence the theory has strict order property and IP.

\section{Urysohn space}\label{sec:Urysohn}


The (complete\footnote{To distinguish with the rational Urysohn space.}) Urysohn space introduced in \cite{Urysohn} is the unique complete separable metric space which contains all separable metric spaces up to isomorphism and is ultra-homogeneous in the sense that any partial isometry between finite subsets can be extended to an isometry onto the whole space. In this section, we study g.e.c.\ graphs on dense subsets of the Urysohn space. Note that there is no obvious defining structure for the Urysphn space, and we will prove by hand that all g.e.c.\ graphs have the same first-order theory on integer-distance free dense subsets. In fact, we will prove that they are all isomorphic. Therefore, the complete Urysohn space is geometrically Rado. 

Let $\mathfrak{U}$ be the complete Urysohn space and $U\subseteq \mathfrak{U}$ be a countable dense subset of it. We call $U$ \emph{integer-distance free} if $d(x,y)\not\in\mathbb{N}$ for all $x\neq y\in U$.

\begin{thm}\label{thm:Ury}
Let $\mathcal{L}=\{C_n:n\in\mathbb{N}^{>0}\}$ be a set of binary relations. Suppose $U_1,U_2\subseteq\mathcal{U}$ are two countable integer-distance free dense subsets. Let $C_n$ be interpreted as $C_n(x,y)$ iff  $n-1<d(x,y)<n$ in $U_1$ and $U_2$ with $d$ the metric coming from $\mathfrak{U}$. Then $U_1$ and $U_2$ are isomorphic as $\mathcal{L}$-structures.
\end{thm}

\begin{proof}
We prove by building a back-and-forth system. Suppose $f:X\to Y$ is a finite partial bijective map witch preserves $C_n$ for all $n\in \mathbb{N}^{>0}$. Given $x_0\in U_1\setminus X$, we want to find $y_0\in U_2\setminus Y$ such that $C_n(x_0,x')$ iff $C_n(y_0,f(x'))$ for all $x'\in X$ and $n\in\mathbb{N}^{>0}$. In the following, we will define a candidate $y$ by defining its distances with all elements in $Y$ and show the triangle inequalities hold, therefore such $y$ exists in $\mathfrak{U}$. Then by density of $U_2$, we find a $y_0$ very close to $y$ that will extend the map $f$.

We first define $y$. Let $D_n=\{f(x'): C_n(x_0,x')\}$ for $n\geq 1$. For any $y'\in D_n$, let $$\epsilon_{y'}:=\min\{d(y'',y')-(n-k)+1: y''\in D_k \text{ with } k<n\}.$$ Write $\min\emptyset=\infty$. Note that $d(y'',y')>n-k-1$ for $y'\in D_n$ and $y''\in D_k$ with $k<n$. Indeed suppose $f(x'')=y''$ and $f(x')=y'$. Then $d(x'',x')\geq d(x_0,x')-d(x_0,x'')>n-1-k$ and as $f$ preserves $C_n$, we get $d(y'',y')>n-k-1$. Thus $\epsilon_{y'}>0$. 
Choose $0<\epsilon<\frac{1}{5}\min\{\epsilon_{y'}:y'\in Y\}$, $\epsilon<1$ and that $$\epsilon<\frac{1}{5}\min\{m-d(y'',y'),d(y'',y')-(m-1):y',y''\in Y\text{ and }m-1<d(y'',y')<m \text{ for some }m\}.$$ 
Note that if $y'\in D_n$ and $y''\in D_k$, then $d(y'',y')<n+k$ as $d(x'',x')\leq d(x_0,x')+d(x_0,x'')<n+k$. Thus, $d(y',y'')<n+k-5\epsilon$.

For $y'\in D_n$ define 
$$
 d(y,y'):=
\begin{cases} 
   n-1+\epsilon_{y'}-2\epsilon, & \text{ if }\epsilon_{y'}<1,\\
   n-\epsilon, & \text{ otherwise}.
   \end{cases}
$$
Note that by definition $n-1<d(y,y')<n$.

Now we prove that $Y\cup\{y\}$ satisfy the triangle inequalities. Consider the triangle $yy_1y_2$ with $y_1\in D_k$ and $y_2\in D_n$ for $n\geq k>0$. There are four cases:
\begin{enumerate}
\item
$\epsilon_{y_1}>1,\epsilon_{y_2}>1$. Then $d(y,y_1)=k-\epsilon$ and $d(y,y_2)=n-\epsilon$. By the definition of $\epsilon$ we have $d(y_1,y_2)<n+k-5\epsilon$. Hence, $$d(y,y_1)+d(y,y_2)=n+k-2\epsilon>n+k-5\epsilon> d(y_1,y_2).$$ We also need to show $d(y_1,y_2)+d(y,y_1)\geq d(y,y_2)$, i.e., $d(y_1,y_2)\geq n-k$. It holds obviously if $n=k$. If $k<n$, then by definition of $\epsilon_{y_2}$ we must have $1<\epsilon_{y_2}\leq d(y_1,y_2)-(n-k)+1$. Therefore, $d(y_1,y_2)>n-k$.
\item
$\epsilon_{y_1}<1,\epsilon_{y_2}<1$. Then $d(y,y_1)=k-1+\epsilon_{y_1}-2\epsilon$ and $d(y,y_2)=n-1+\epsilon_{y_2}-2\epsilon$. By definition, there are $y_3\in D_{k'}, y_4\in D_{n'}$ with $k'<k$ and $n'<n$ such that $d(y_1,y_3)=k-k'-1+\epsilon_{y_1}$ and $d(y_2,y_4)=n-n'-1+\epsilon_{y_2}$. Note that $d(y_3,y_4)< k'+n'-5\epsilon$. Then 
\begin{align*}
&d(y,y_1)+d(y,y_2)-d(y_1,y_2)\geq d(y,y_1)+d(y,y_2)-(d(y_1,y_3)+d(y_3,y_2))\\
&\geq d(y,y_1)+d(y,y_2)-d(y_1,y_3)-d(y_3,y_4)-d(y_4,y_2)\\
&=n+k-2+\epsilon_{y_1}+\epsilon_{y_2}-4\epsilon-(n+k-2+\epsilon_{y_1}+\epsilon_{y_2}-n'-k')-d(y_3,y_4)\\
&=n'+k'-4\epsilon-d(y_3,y_4)> n'+k'-4\epsilon-n'-k'+5\epsilon>0.
\end{align*}
Note $k'<k\leq n$. Therefore, by the definition of $\epsilon_{y_2}$ we have $d(y_2,y_3)\geq n-k'-1+\epsilon_{y_2}$.
\begin{align*}
&d(y_1,y_2)\geq d(y_2,y_3)-d(y_1,y_3)\geq n-k'-1+\epsilon_{y_2}-(k-k'-1+\epsilon_{y_1})\\
&=n+\epsilon_{y_2}-k-\epsilon_{y_1}=d(y,y_2)-d(y,y_1).
\end{align*}
For the last triangle inequality, note that if $d(y,y_2)\geq d(y,y_1)$ then it is obviously true. If not, then $k=n$ (because $\epsilon_{y_1} < 1$) and we may switch $y_1$ and $y_2$ and the same proof goes through.
\item
$\epsilon_{y_1}<1$, $\epsilon_{y_2}>1$. Then $d(y,y_1)=k-1+\epsilon_{y_1}-2\epsilon$ and $d(y,y_2)=n-\epsilon$. Note that $d(y,y_1)<d(y,y_2)$ as $\epsilon_{y_1}<1$. There is $y_3\in D_{k'}$ with $k'<k$ such that $d(y_1,y_3)=k-k'-1+\epsilon_{y_1}$. Note that $d(y_2,y_3)-(n-k')+1\geq \epsilon_{y_2}>1$. Hence, $d(y_2,y_3)>n-k'$. By the definition of $\epsilon$, we have $d(y_2,y_3)>n-k'+5\epsilon$. Also $d(y_2,y_3)< n+k'-5\epsilon$. Hence,
\begin{align*}
&d(y_1,y_2)\leq d(y_2,y_3)+d(y_3,y_1)< n+k'-5\epsilon+k-k'-1+\epsilon_{y_1}\\
&=(n-\epsilon)+(k-1+\epsilon_{y_1}-2\epsilon)-2\epsilon<d(y,y_2)+d(y,y_1).
\end{align*}
And 
\begin{align*}
&d(y_1,y_2)+d(y,y_1)\geq d(y_2,y_3)-d(y_3,y_1)+d(y,y_1)\\
&>n-k'+5\epsilon-(k-k'-1+\epsilon_{y_1})+k-1+\epsilon_{y_1}-2\epsilon\\
&=n+3\epsilon>d(y,y_2).
\end{align*}
\item
$\epsilon_{y_1}>1$, $\epsilon_{y_2}<1$. Then $d(y,y_1)=k-\epsilon$ and $d(y,y_2)=n-1+\epsilon_{y_2}-2\epsilon$. Note that we may assume $k<n$, hence $d(y,y_1)<d(y,y_2)$. Otherwise $k=n$ and by switching $y_1$ and $y_2$, we are in the previous case. Therefore, by the definition of $\epsilon_{y_2}$ we have $d(y_1,y_2)\geq n-k-1+\epsilon_{y_2}$. Also there is  $y_4\in D_{n'}$ with $n'<n$ such that $d(y_2,y_4)=n-n'-1+\epsilon_{y_2}$. Note that $d(y_1,y_4)<k+n'-5\epsilon$.
\begin{align*}
&d(y_1,y_2)\leq d(y_1,y_4)+d(y_2,y_4)<k+n'-5\epsilon+n-n'-1+\epsilon_{y_2}\\
&=(k-\epsilon)+(n-1+\epsilon_{y_2}-2\epsilon)-2\epsilon\\
&=d(y,y_1)+d(y,y_2)-2\epsilon<d(y,y_1)+d(y,y_2).
\end{align*}
And 
$d(y_1,y_2)+d(y,y_1)\geq n-k-1+\epsilon_{y_2}+k-\epsilon=n-1+\epsilon_{y_2}-\epsilon>d(y,y_2)$.
\end{enumerate}

By the property of Urysohn space, there is $y\in \mathfrak{U}$ such that $d(y,y')$ has the desired value for all $y'\in Y$. In particular $n-1<d(y,y')<n$ for all $y'\in D_n$. Now we want to find $y_0\in U_2$ that also satisfy the same property. By the definition, either $d(y,y')=n-1+\epsilon_{y'}-2\epsilon>n-1+3\epsilon$ (as $5\epsilon<\epsilon_{y'}$) or $d(y,y')=n-\epsilon$. Therefore, $n-1+\epsilon/2<d(y,y')<n-\epsilon/2$ for $y'\in D_n$. By density, there is a point $y_0\in U_2$ such that $d(y_0,y)<\epsilon/4$. Then $d(y_0,y')<d(y,y')+d(y,y_0)<n-\epsilon/4$ and $d(y_0,y')>d(y,y')-d(y,y_0)>n-1+\epsilon/4$ for all $y'\in D_n$. Now we may extend $f$ by sending $x_0$ to $y_0$ and we are done.
\end{proof}

\begin{coro}\label{cor:UrysohnRado}
The Uryshohn space is geometrically Rado on any countable integer-distance free dense set. More precisely, let $G_1=(U_1,E_1)$ and $G_2=(U_2,E_2)$ be two g.e.c.\ graphs on countable integer-distance free dense sets $U_1,U_2\subseteq \mathfrak{U}$ respectively. Then $G_1$ and $G_2$ are isomorphic.
\end{coro}
\begin{proof}
We build a graph isomorphism by building a back-and-forth system that preserves both the edge relation and $\{C_n:n\in\mathbb{N}^{>0}\}$. Suppose $f:X\to Y$ is such a finite map and $x_0\in U_1$. We want to extend $f$ to $X\cup\{x_0\}$. We will modify the proof of \cref{thm:Ury}. We know there is $\epsilon>0$ and $y'\in \mathcal{U}$ such that for all $y_0\in U_1\cap B_{\epsilon/2}(y')$, for all $n>0$, for all $y\in Y$ with $n-1<d(x_0,f^{-1}(y))<n$, we have $n-1<d(y_0,y)<n$. Let $Q:=\{y\in Y:x_0\mathrel{E_1}f^{-1}(y)\}$. Note that $Q\subseteq B_1(y')$ as $f^{-1}(Q)\subseteq B_1(x_0)$ and $B_1$ is interdefinable with $C_1$. By the g.e.c.\ property of $G_2$, there is $y_0\in U_1\cap B_{\epsilon/2}(y')$ such that $y_0\mathrel{E_2}y$ iff $y\in Q$ and we are done.
\end{proof}

\begin{remark}
By the proofs above we can further conclude the following. If $G$ is a g.e.c.\ graph on a countable dense integer-distance free set in the Urysohn space $\mathfrak{U}$. And let $\mathcal{L}':=\mathcal{L}_G\cup\{C_n:n\in\mathbb{N}^{>0}\}$. Then
\begin{itemize}
\item
$G$ is homogeneous as an $\mathcal{L}'$-structure;
\item
$Th_{\mathcal{L}'}(G)$ has quantifier elimination.
\end{itemize}

\end{remark}

\bibliographystyle{alpha}
\bibliography{bibiliography}

\end{document}